\documentclass[12pt]{amsart}
\usepackage{amscd,amssymb,graphicx,color,a4wide,hyperref,cite}
\usepackage{listings}
\usepackage{color}

\definecolor{dkgreen}{rgb}{0,0.6,0}
\definecolor{gray}{rgb}{0.5,0.5,0.5}
\definecolor{mauve}{rgb}{0.58,0,0.82}

\usepackage{hyperref}
\usepackage{url}

\usepackage{mathrsfs}

\usepackage{epstopdf}

\usepackage[all]{xy}
\newtheorem{theorem}{Theorem}[section]
\newtheorem{lemma}[theorem]{Lemma}
\newtheorem{proposition}[theorem]{Proposition}
\newtheorem{corollary}[theorem]{Corollary}
\newtheorem{conjecture}[theorem]{Conjecture}

\theoremstyle{definition}
\newtheorem{algorithm}[theorem]{Algorithm}
\newtheorem{definition}[theorem]{Definition}
\newtheorem{construction}[theorem]{Construction}

\newtheorem{example}[theorem]{Example}

\theoremstyle{remark}
\newtheorem{remark}[theorem]{Remark}

\numberwithin{equation}{section}
\numberwithin{figure}{section}

\newcommand{\ZZ} {\mathbb{Z}}

\newcommand{\RR} {\mathbb{R}}

\newcommand{\PP} {\mathbb{P}}

\renewcommand{\AA} {\mathbb{A}}

\newcommand {\shF}  {\mathcal{F}}
\newcommand {\shG}  {\mathcal{G}}
\newcommand {\shH}  {\mathcal{H}}

\newcommand {\shL}  {\mathcal{L}}

\newcommand {\shV}  {\mathcal{V}}
\newcommand {\shX}  {\mathcal{X}}

\newcommand {\Aff}  {\operatorname{Aff}}

\newcommand {\codim} {\operatorname{codim}}

\newcommand {\coker} {\operatorname{coker}}

\renewcommand{\emptyset}{\varnothing}

\newcommand {\Ext}  {\operatorname{Ext}}

\newcommand {\GL}  {\operatorname{GL}}

\newcommand {\Hom}  {\operatorname{Hom}}

\newcommand {\Int}  {\operatorname{Int}}

\newcommand {\lra}  {\longrightarrow}

\renewcommand{\O}  {\mathcal{O}}

\newcommand {\rank} {\operatorname{rank}}

\newcommand {\Square}  {\operatorname{Sq}}

\renewcommand {\tilde} {\widetilde}

\newcommand {\U} {\mathcal{U}}

\def\mydate{\ifcase\month \or January\or February\or March\or
April\or May\or June\or July\or August\or September\or October\or 
November\or December\fi \space\number\day,\space\number\year}

\graphicspath{{images/}}

\newcommand\restr[2]{{
  \left.\kern-\nulldelimiterspace 
  #1 
  \vphantom{\big|} 
  \right|_{#2} 
  }}

\begin{document}

\title[Real Lagrangian]{On the cohomology groups of real Lagrangians in Calabi--Yau threefolds}

\author{H\"ulya Arg\"uz}
\address{Laboratoire de Math\'ematiques, Universit\'e de
Versailles St Quentin en Yvelines, France}
\email{nuromur-hulya.arguz@uvsq.fr}

\author{Thomas Prince}
\address{Mathematical Institute\\University of Oxford\\Woodstock Road\\OX$2$ $6$GG\\UK}
\email{thomas.prince@magd.ox.ac.uk}

\date{\today}

\begin{abstract}
The quintic threefold $X$ is the most studied Calabi-Yau $3$-fold in the mathematics literature. In this paper, using  \u{C}ech-to-derived spectral sequences, we investigate the mod $2$ and integral cohomology groups of a real Lagrangian $\breve{L}_\RR$, obtained as the fixed locus of an anti-symplectic involution in the mirror to $X$. We show that $\breve{L}_\RR$ is the disjoint union of a $3$-sphere and a rational homology sphere. Analysing the mod $2$ cohomology further, we deduce a correspondence between the mod $2$ Betti numbers of $\breve{L}_\RR$ and certain counts of integral points on the base of a singular torus fibration on $X$. By work of Batyrev, this identifies the mod $2$ Betti numbers of $\breve{L}_\RR$ with certain Hodge numbers of $X$. Furthermore, we show that the integral cohomology groups $H^j(\breve{L}_\RR,\ZZ)$ of $\breve{L}_\RR$ are $2$-primary for $j \neq 0,3$; we conjecture that this holds in much greater generality.

\end{abstract}
\maketitle

\tableofcontents

\section{Introduction}
\label{sec:intro}
Mirror symmetry proposes the existence of mirror pairs of
Calabi--Yau threefolds $(X,\breve{X})$, which fulfil the Hodge theoretic relationship
\begin{equation}
\label{Eq: Hodge theoretic}    
H^q(X,\Omega^p_X) \cong H^q(\breve{X},\Omega^{n-p}_{\breve{X}}). 
\end{equation}
The mirror correspondence of Batyrev--Borisov \cite{BB} constructs mirror pairs $(X,\breve{X})$ as anti-canonical hypersurfaces in four-dimensional toric varieties defined by dual lattice polytopes. In this paper, we investigate a relationship between (mod $2$) Betti numbers of certain real Lagrangians $\breve{L}_\RR \subset \breve{X}$ and integral points in the lattice polytope $P$ that defines~$X$. In view of Batyrev's result~\cite{B} relating the Hodge numbers of $X$ to integral points in $P$, this suggests an intriguing correspondence between the mod $2$ topology of real Lagrangians in Calabi--Yau threefolds and Hodge numbers of the mirror Calabi--Yau.

In our setting, the choice of an integral affine structure with simple singularities on the boundary  
$B := \partial P$ of the lattice polytope $P$ specifies dual singular Lagrangian torus fibrations $f \colon X \to B$ and $\breve{f} \colon \breve{X} \to B$.   
These fibrations were constructed by Gross \cite{GrossTopology,GrossBB} and Casta\~{n}o-Bernard--Matessi \cite{CBM3}; they give topological versions of the structures predicted by the Strominger--Yau--Zaslow Conjecture~\cite{SYZ}. 
Note that in a Batyrev-Borisov mirror pair $(X, \breve{X})$, the mirror $\breve{X}$ is only well-defined up to birational 
modifications, but that the choice of the integral affine structure on $B$ leads to the construction of a specific 
torus fibration $f \colon X \to B$, and then, by compactification of the dual torus fibration, of a specific topological model $\breve{X}$ of the mirror.
The real Lagrangian $\breve{L}_\RR \subset \breve{X}$ that we consider is the fixed locus of an anti-symplectic involution of $\breve{X}$ given on smooth fibres of $\breve{f}$ by $x \mapsto -x$. This anti-symplectic involution 
was introduced in \cite{CBMS} and studied in 
\cite{CBM} .

We now restrict attention to the case in which $X$ is the quintic threefold hypersurface in 
$\mathbb{P}^4$. We will consider a specific integral affine manifold with simple singularities on 
$B$, the corresponding Lagrangian torus fibration 
$f \colon X \to B$ and the corresponding topological model $\breve{X}$ of the mirror, coming with the Lagrangian torus fibration $\breve{f} \colon \breve{X} \to B$. In earlier work \cite{AP19}, we showed that the real Lagrangian $\breve{L}_\RR \subset \breve{X}$ is the disjoint union of a $3$-sphere and a multi-section $\breve{\pi} \colon \breve{\shL_\RR} \to B$ of $\breve{f}$, and computed the mod $2$ Betti numbers:
\begin{equation}
\label{eq:Betti_numbers}
h^0(\breve{\mathcal{L}}_{\RR},\ZZ_2)=h^3(\breve{\mathcal{L}}_{\RR},\ZZ_2)=1, \,\ \,\ \,\    h^1(\breve{\mathcal{L}}_{\RR},\ZZ_2)=h^2(\breve{\mathcal{L}}_{\RR},\ZZ_2)=101.
\end{equation}
In \S\ref{sec:mod2_cohomology} below we use a \u{C}ech-to-derived spectral sequence which relates the mod $2$ cohomology groups of $\breve{L}_\RR$ to the set of integral points in the reflexive polytope $P$ that defines $X$.

\begin{theorem}[See Theorem~\ref{thm:cech_to_derived_Z2} below]
    \label{thm:cech_to_derived_Z2_in_intro}
	The \u{C}ech-to-derived spectral sequence, relative to the open cover $\U$ defined in Construction~\ref{def:coarse_cover}, for the sheaf $\breve{\pi}_\star\ZZ_2$ has the $E_2$ page
	\[
	\xymatrix@R-2pc{
		\ZZ_2^{60} & & & \\
		\ZZ_2^{100} & \ZZ_2^{40} & & \\
		\ZZ_2 & \ZZ_2 & \ZZ_2 & \ZZ_2
	}
	\]
	and degenerates at the $E_2$ page.
\end{theorem}

Analysing the $E_2$ page of the spectral sequence in Theorem~\ref{thm:cech_to_derived_Z2_in_intro} more closely, we show in Theorem~\ref{Relating to integral points} that for each $p \in \{1,2\}$ there is a canonical choice of basis for $E_2^{2-p,p}$, which is in bijection with the set of integral points contained in relative interiors of $p$-dimensional faces of $P$. Using Batyrev's formula \eqref{Eq: Batyrev Hodge}, which expresses Hodge numbers in terms of such integral points, we deduce that
\begin {equation}
\label{eq: Hodge numbers and spectral sequences}
h^{2,1}(X) = \dim E_2^{0,2} + \dim E_2^{1,1} + 1=h^1(\breve{L}_\RR,\ZZ_2).
\end{equation}
We expect this correspondence to extend to all anti-canonical hypersurfaces in smooth toric Fano fourfolds: see Conjecture~\ref{conj:generalised_example}. 

In \S\ref{sec:betti_numbers} we study also the integral cohomology of $\breve{\shL}_{\RR}$, using a \u{C}ech-to-derived spectral sequence for the sheaf $\breve{\pi}_\star\ZZ$.
\begin{theorem}[See Theorem~\ref{thm:cech_to_derived} below]
    \label{thm:cech_to_derived_in_intro}
	The \u{C}ech-to-derived spectral sequence, relative to the open cover $\U$ defined in Construction~\ref{def:coarse_cover}, for the sheaf $\breve{\pi}_\star\ZZ$ has the $E_2$ page
	\[
	\xymatrix@R-2pc{
		\ZZ_2^{60} & & & \\
		0 & \ZZ_2^{36} \oplus \ZZ_4^6 \oplus \ZZ_8^4 \oplus \ZZ^2_{32} & & \\
		\ZZ & 0 & \ZZ_2 & \ZZ
	}
	\]
	and degenerates at the $E_2$ page.
\end{theorem}

This implies that
\[
H^0(\breve{\shL}_\RR,\ZZ) \cong H^3(\breve{\shL}_\RR,\ZZ) \cong \ZZ, \,\ \,\ H^1(\breve{\shL}_\RR,\ZZ) \cong 0, \,\ \,\ H^2(\breve{\shL}_\RR,\ZZ) \cong T,
\]
where $T$ is a $2$-primary finite abelian group (that is, the order of each element of $T$ is a power of $2$) such that every element has order less than or equal to $2^7$. In particular, $\breve{\shL}_\RR$ is a rational homology sphere.

In \S\ref{sec:heegaard_splitting} we give an alternative, more topological, approach to some of these calculations, which applies in much greater generality. Consider once again the case where $X$ and $\breve{X}$ form a Batyrev--Borisov mirror pair. As in the quintic case, the real Lagrangian $\breve{L}_\RR \subset \breve{X}$ is the disjoint union of a $3$-sphere and a non-trivial component $\breve{\shL}_\RR$. We construct a Heegaard splitting of $\breve{\shL}_\RR$ and explain (in Algorithm~\ref{alg:fundamental_group}) how to use this splitting to compute $\pi_1(\breve{\shL}_\RR)$. In particular, since this determines $H^1(\breve{\shL}_\RR,\ZZ)$, this verifies parts of the \u{C}ech-to-derived calculations.

We view Theorem~\ref{thm:cech_to_derived_in_intro} as experimental evidence for the following conjecture on the structure of the integral cohomology groups of the real Lagrangians $\breve{L}_\RR$ in a more general setting.

\begin{conjecture}
\label{Conj: $2$-primary} Let $X$ be Calabi-Yau 3-fold obtained as a (crepant resolution of a) complete 
intersection in a toric Fano variety. Let $f \colon X \rightarrow B$ be a Lagrangian torus fibration constructed as in 
\cite{GrossTopology,GrossBB, CBM3} and let $L_{\RR}$ be the real Lagrangian in $X$ obtained as fixed point locus of the anti--symplectic involution constructed in \cite{CBMS} and given on smooth fibers of $f$ by $x \mapsto -x$.
Then, the cohomology groups $H^j(L_\RR, \ZZ)$ are $2$-primary for $0<j<3$.
\end{conjecture}
We verify Conjecture~\ref{Conj: $2$-primary} for the Schoen's Calabi--Yau in \cite{AP:Schoen}. Finally, we note that Theorem~\ref{thm:cech_to_derived_in_intro} has applications to mirror symmetry. Homological Mirror Symmetry predicts the existence of a rank seven sheaf $F$ on the quintic threefold $X$, associated to the Lagrangian $\breve{\shL}_\RR$. Moreover, assuming that $\breve{\shL}_\RR$ bounds no holomorphic discs, Theorem~\ref{thm:cech_to_derived_in_intro} implies that $F$ is a \emph{spherical object} in the derived category of $X$ which is orthogonal to the structure sheaf, that is, $\Ext^i(\O_X,F) = \Ext^i(F,\O_X) = 0$ for all $i \in \ZZ$. Investigating sheaves mirror to the real Lagrangians we study will be the focus of future work. 


\subsection*{Related work}
The mod $2$ cohomology of real Lagrangians has been extensively studied in the literature from several points of view, including equivariant cohomology~\cite{biss2004mod2}, real algebraic geometry~\cite{bihan2003asymptotic, finashin2019first} and tropical geometry~\cite{itenberg-topology,itenberg2006asymptotically}. Furthermore, the relation between Hodge numbers and the $\ZZ_2$ cohomology groups of real Lagrangians can also be studied using tropical homology which is introduced in \cite{itenberg2019tropical}, and whose real analog is studied in \cite{renaudineau2018bounding}. We also note that Lagrangian submanifolds in the mirror quintic have been constructed using tropical geometry in~\cite{MR19}.

In the present paper, we investigate following 
\cite{GrossTopology,GrossBB,CBMS, CBM3} a real Lagrangian defined as fixed point locus of an anti-symplectic involution acting on a Lagrangian torus fibration constructed from an integral affine manifold with simple singularities. An a priori different construction due to Gross and Siebert \cite{GS1,GS2}, starting with an integral affine manifold with simple singularities and additional ``gluing data", produces toric degenerations of algebraic Calabi-Yau varieties. When these ``gluing data" is real \cite[Cor 7.2]{AS}, we obtain families of algebraic Calabi-Yau varieties endowed with a real algebraic structure, i.e.\ the data of an anti-holomorphic involution.
The topology of the real loci of these families has been studied in  
\cite{AS} using Kato-Nakayama spaces of the special fiber endowed with log structure.

A comparison of the topological torus fibrations of \cite{GrossTopology,GrossBB,CBMS}
with the topological torus fibrations on the Kato-Nakayama spaces of toric degenerations has been announced in 
\cite{RZ3}. A natural question to ask is if it is also possible to compare the anti-symplectic involution of 
\cite{CBMS}, which is our object of study in the case of the mirror quintic, with the real structure on Kato-Nakayama spaces studied in \cite{AS}. Kato--Nakayama spaces admit a natural map to the unit circle $S^1$, and we expect that the real Lagrangian obtained as fixed locus of the anti-symplectic involution
should be homeomorphic to the real locus of the fibre of the Kato-Nakayama space over $1\in S^1$ defined using trivial gluing data (see \cite[Remark 4.3]{AS}). Nonetheless, so far this question, namely whether the fixed point locus of anti-symplectic involutions of \cite{CBM} would agree with the real locus in the phase of the Kato--Nakayama space, remains open. We note that the techniques recently introduced in \cite{RZ1, RZ3} should be useful for answering it.

\subsection*{Acknowledgements}
We thank Mark Gross and Tom Coates for many useful and inspiring conversations. 
We also thank the anonymous referees for their valuable comments 
which helped to improve the manuscript.
This project has received funding from the European Research Council (ERC) under the European Union’s Horizon 2020 research and innovation programme (grant agreement No. 682603). H.A. was supported by Fondation Math\'ematique Jacques Hadamard. TP was partially supported by a Fellowship by Examination at Magdalen College, Oxford.


\section{SYZ mirror symmetry from a topological perspective}
\label{sec:SYZ Mirrors}

Mirror symmetry is a phenomenon which relates the geometry of pairs of Calabi--Yau varieties, and was first discovered in the string theory literature by Greene--Plesser
\cite{GP90} and Candelas--Lynker--Schimmrigk \cite{CLS90}, confirming earlier suggestions of Dixon \cite{D88} and Lerche--Vafa--Warner
\cite{LVW89}. This phenomenon has famously been used to compute enumerative invariants of Calabi--Yau varieties, see the seminal work of Candelas--de la Ossa--Greene--Parkes~\cite{COGP}. In the last twenty-five years several techniques relating complex and symplectic geometries of Calabi-Yau's have
been developed, suggesting deep connections between the geometries of the underlying mirror manifolds. One of the central proposals in this context has been formulated by Strominger--Yau--Zaslow~\cite{SYZ}, the \emph{SYZ Conjecture}. This suggests that, roughly speaking, mirror Calabi-Yau's arise as dual special Lagrangian torus fibrations. Partial verifications of the SYZ conjecture have been established in various contexts \cite{HT,GrossSYZ, ChanSYZ}. 

A topological version of the SYZ conjecture, focusing on torus fibrations without a reference to the special Lagrangian condition, has been investigated by Zharkov \cite{Z00} and Gross~\cite{GrossTopology}. For Gross, a \emph{topological torus fibration} on a Calabi--Yau $X$ is a continuous proper map $f \colon X \to B$ with connected fibres between topological manifolds whose general fibres are tori. The base $B$ here is an integral affine manifold with singularities.



\begin{definition}
Given an $n$-dimensional manifold, an \emph{integral affine structure} is given by an open cover $\{ \mathcal{U}_i\}$ for it, along with coordinate charts $\psi_i\colon\mathcal{U}_i \to M_{\RR}$, with transition functions in
\[
\Aff(M)= M \rtimes GL_n(\ZZ),
\]
where $M$ denotes a free abelian group of rank $n$, and $M_{\RR}:= M \otimes_{\ZZ} \RR$ is the associated real vector space.
We call a topological manifold an \emph{integral affine manifold with singularities} if there exists a union of submanifolds $\Delta \subset B$, of codimension at least $2$, such that
\[
B_0 := B\setminus \Delta
\]
is an integral affine manifold. We refer to $\Delta$ as the \emph{discriminant locus of the affine structure}. 
\end{definition}

We recall that the smooth locus $B_0$ of an integral affine manifold is the base of a pair of torus fibrations, obtained as the quotients of $TB_0$ and $T^\star B_0$ by the lattice of integral tangent vectors $\Lambda$ and covectors $\breve{\Lambda}$ respectively. To investigate topological properties of mirror pairs, such as Hodge theoretic dualities, Gross \cite{GrossTopology,GrossSLagI} restricts to Calabi--Yau manifolds admitting topological torus fibrations over integral affine manifolds with \emph{simple} singularities. These fibrations can be obtained canonically as compactifications of the non-singular torus fibrations over $B_0$ described above once certain assumptions on the affine monodromy around the discriminant locus are imposed. We recall that the affine monodromy is defined as follows.

\begin{definition}
Fix a point $b\in B_0$. Let $\gamma\colon S^1\rightarrow B_0$ be a loop based at $b$ and let $U_1,\ldots,U_m$ be a finite cover of the image of $\gamma$ by charts of the integral affine structure on $B_0$. 
Denote by $A_{i,i+1}^{-t}$ the inverse transpose of the linear part of the change of coordinate function defined on $U_{i}\cap U_{i+1}$. The \emph{affine monodromy representation} $\psi\colon \pi_1(B_0,b) \to \GL_n(\ZZ)$ is defined as 
\[
\psi =
\left\{
	\begin{array}{ll}
		A_{1,m}^{-t} \cdots A_{2,1}^{-t}  & \mbox{if } m \geq 2 \\
		\mathrm{Id} & \mbox{otherwise} 
	\end{array}
\right.
\]
\end{definition}

Note that the definition of affine monodromy is independent of the representative $\gamma \in \pi_1(B_0,b)$, and thus it is a well-defined homomorphism, see \cite{A,AM}. The affine monodromy is, by definition, the inverse transpose of the linear part of the standard monodromy representation around a loop $\gamma$ in $B_0$, which we denote $T_\gamma$, see \cite[Definition~$1.4$]{GS1}. The following theorem, which establishes the existence of topological Calabi--Yau compactifications over integral affine manifolds with simple singularities, is one of the main results of \cite[\S$2$]{GrossTopology}.

\begin{theorem}
\label{Thm:Marks}
Let $B$ be a $3$-manifold and let $B_0 \subseteq B$ be a
dense open set such that $\Delta := B \setminus B_0$ is a trivalent graph, with a partition on its set of vertices of into positive and negative vertices. Suppose that there is a $T^3$-bundle
\begin{equation}
f_0 \colon X_0 \to B_0
\end{equation} 
such that its local monodromy is generated by the following matrices:
\begin{itemize}
    \item[1)] Around any edge of $\Delta$, the monodromy is given by 
\begin{align}
\label{Eqn: monodromy around edge}
T  =  \left( \begin{matrix} 1&0&0\\ 1&1&0\\0&0&1\\ \end{matrix}\right).
\end{align}
\item[2)] Around any negative vertex of $\Delta$ the monodromy is given by
\begin{align}
\label{Eqn: monodromy around negative}
T_1  =  \left( \begin{matrix} 1&1&0\\ 0&1&0\\0&0&1\\ \end{matrix}\right), \,\ T_2  =  \left( \begin{matrix} 1&0&-1\\ 0&1&0\\0&0&1\\ \end{matrix}\right), \,\ T_3  =  \left( \begin{matrix} 1&-1&1\\ 0&1&0\\0&0&1\\ \end{matrix}\right).
\end{align}
\item[3)] Around any positive vertex of $\Delta$ the monodromy is given by
\[
(T_1^t)^{-1},(T_2^t)^{-1},(T_3^t)^{-1}.
\]
\end{itemize}
Then, there is a compactification of $f_0$, to a topological torus fibration $f\colon X \to B$ such that the singular fibers are as follows.
\begin{itemize}
    \item[1)] For any point $p$ in the interior of an edge of the discriminant locus, $f^{-1}(p)$ is homeomorphic to the product of a nodal elliptic curve with $S^1$.
    \item[2)] For any negative vertex $v_- \subset \Delta$, $f^{-1}(v_-)$ admits a map to a two-dimensional torus with $S^1$ fibres away from a figure eight, over which the fibres are single points, as described in \cite[Chapter~$6.4$]{DBranes09}, and $\chi(f^{-1}(v_-)) = -1$.	
    \item[3)] For any positive vertex $v_+ \subset \Delta$, $f^{-1}(v_+)$ is homeomorphic to a three dimensional analogue of a nodal elliptic curve as described in \cite[Example~$2.6$(5)]{GrossTopology} and $\chi(f^{-1}(v_+)) = 1$. 
\end{itemize}
\end{theorem}

Among the topological Calabi--Yau compactifications $f\colon X \to B$ obtained by Theorem \ref{Thm:Marks} is one of the most extensively studied examples in the literature: the quintic threefold. This example is described in detail in \cite[Theorem~$0.2$]{GrossTopology}. Such topological compactifications exist for a wide range of Calabi--Yau threefolds, for instance for Calabi--Yau complete intersections in toric varieties, as described in \cite{HZci, GrossBB}.

It is shown in \cite{CBM3} that these topological compactifications can be carried out in the symplectic category. 
More precisely, it is shown in \cite{CBM3} that there exists a smooth symplectic structure on $X$ such that the a small perturbation of the topological fibration $f \colon X \to B$ which appears in Theorem~\ref{Thm:Marks} becomes piecewise-smooth Lagrangian. The small perturbation replaces the discriminant locus $\Delta$ by a small thickening around its negative vertices. These Lagrangian torus fibrations admit a Lagrangian section and it is shown in \cite{CBMS} that there exists a unique anti-symplectic involution of $X$, preserving the Lagrangian fibration and fixing 
the Lagrangian section. The fixed point locus of this anti-symplectic involution is a real Lagrangian 
in $X$. The main topic of the present paper is the study of the topology of this real Lagrangian for $X$ a 
specific topological model of the mirror quintic. From the topological point of view, the small deformation of $f$
leading to a small thickening of the discriminant locus is irrelevant, and so we will allow ourselves the abuse of 
language to call $f$ a Lagrangian fibration and to study the topology of $f$ rather of a small
perturbation of $f$.



\section{The mirror to the quintic threefold $X$}
\label{sec:integral_affine_manifolds}

In this section we describe topological Calabi--Yau compactifications 
\[
f \colon X \lra B \,\ \,\ \mathrm{and} \,\ \,\ \breve{f} \colon \breve{X} \lra B
\]
on the quintic threefold $X$ and its mirror $\breve{X}$, where $B$ is an integral affine manifold with simple singularities homeomorphic to $S^3$.
Denoting by $B_0 \subset B$ the smooth locus of the integral affine structure, the spaces  $X$ and $\breve{X}$
are compactifications
respectively of $T^{*}B_0/\breve{\Lambda}$ 
and $TB_0/\Lambda$, where $\Lambda$ is the local system of integral tangent vectors on $B_0$.
For further details we refer to \cite[\S$19.3$]{GHJ}, and \cite[Example~$6$]{CBM}.

Let $\Delta_{\PP^4}$ be the moment polytope of the toric variety $\PP^4$, given by the $4$-simplex obtained as the convex hull of the set
\begin{align}
\label{eq: points}    
\shV := \{ (-1,-1,&-1,-1), (4,-1,-1,-1), (-1,4,-1,-1), \\ \nonumber &(-1,-1,4,-1), (-1,-1,-1,4) \}.
\end{align}
We set $P_i$ denote the $i$th member of $\shV$ for each $i \in \{1,\ldots,5\}$. Moreover, we set $B := \partial \Delta_{\PP^4}$, the boundary of $\Delta_{\PP^4}$, which is homeomorphic to the $3$-sphere. Thus $B$ is the union of five tetrahedra, glued pairwise to each other along a common triangular face, as illustrated in Figure~\ref{Base}, together with some positive and negative vertices of the discriminant locus. Note that $B$ contains ten triangular faces, ten edges, and five vertices.

Let $\sigma_{ijk}$ denote the triangular face of $\Delta_{\PP^4}$ spanned by the vertices $P_i,P_j,P_k \in \shV$. Fix a regular triangulation of $\sigma_{ijk}$, displayed in black in Figure~\ref{fig:the_discriminant_locus}, such that the vertices of this triangulation are the integral points of $\sigma_{ijk}$. Let $\Delta_{ijk}$ be the union of the one dimensional cells in the first barycentric subdivision of this triangulation which do not contain an integral point of $\sigma_{ijk}$. We illustrate $\Delta_{ijk}$ in Figure~\ref{fig:the_discriminant_locus} in red. 
\begin{figure}
\resizebox{0.3\textwidth}{!}{
\includegraphics{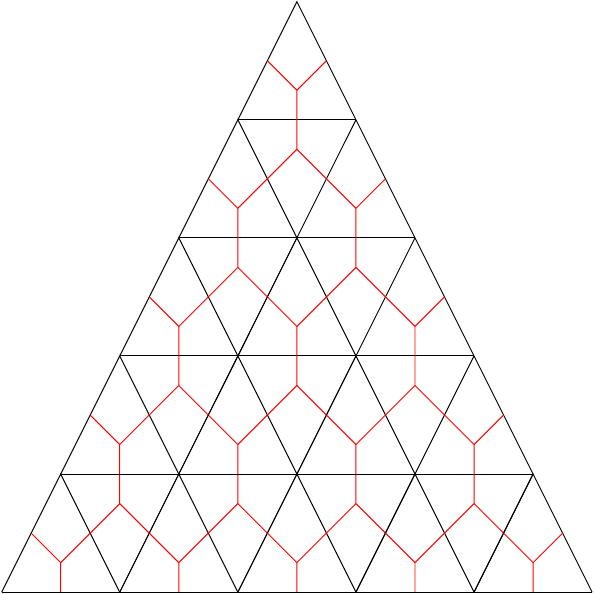}
}
\caption{The discriminant locus on a triangular face.}
\label{fig:the_discriminant_locus}
\end{figure}
Finally we set the discriminant locus
\begin{equation}
\label{TheDiscriminantLocus}
\Delta:= \bigcup_{i,j,k} \Delta_{ijk}
\nonumber
\end{equation}
The affine structure on $B_0 := B\setminus \Delta$, for the fibration on the quintic threefold
\[
f \colon X \lra B
\]
is described in \cite[p.~157]{GHJ}. Note that there are two sorts of vertices of $\Delta$: the vertices that are in the interior of a triangular face, which are \emph{negative vertices}; and the vertices that are in the interiors of edges, which are \emph{positive vertices}, following the convention used in \cite[p.~$241$]{CBM}. The local monodromy of the affine structure around each type of vertex is described as in Theorem \ref{Thm:Marks}, and in more detail in \cite[Appendix~A, Example~A.$1$]{AP19}.


\section{The real Lagrangian $\breve{L}_{\RR}$ in $\breve{X}$}
\label{sec:real_locus}

In this section we describe a real Lagrangian $\breve{L}_{\RR} \subset \breve{X}$ in the mirror to the quintic threefold, 
\[
\breve{f} \colon \breve{X} \lra B
\]
as described in \S\ref{sec:integral_affine_manifolds}.
For details we refer to \cite[Section $2$]{CBM} and \cite{AP19}. We investigate the topology of $\breve{L}_\RR$ further, and show that each of its connected components is orientable in Proposition \ref{pro:orientable}.

\begin{figure}
\resizebox{0.8\textwidth}{!}{
\includegraphics{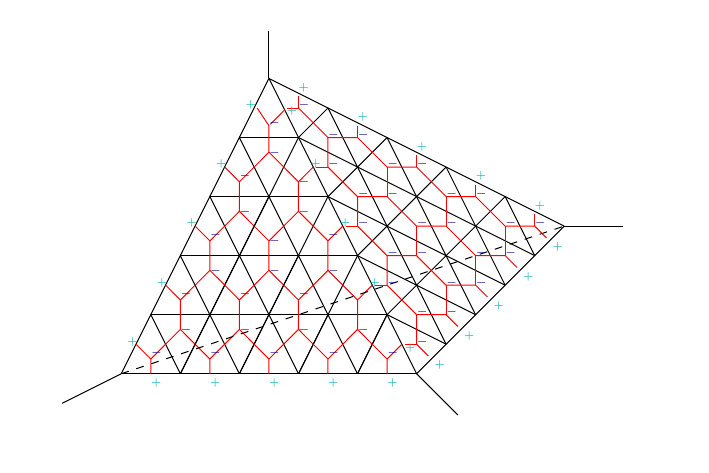}
}
\caption{Positive and negative vertices of $\Delta$, displayed on one of the five tetrahedra forming the base $B=\partial \Delta_{\PP^4}$ of the fibration on the quintic threefold. The edges emanating from the four vertices are parts of the edges of the other four tetrahedra.}
\label{Base}
\end{figure}

The space $\breve{L}_\RR \subset \breve{X}$ is the fixed point set of an anti-symplectic involution which acts on each fibre of $\breve{f}_0 \colon \breve{X}_0 \to {B_0}$ by taking $x \mapsto -x$. Note that this involution extends over fibres $\breve{f}^{-1}(p)$ for $p \in \Delta$, see \cite{CBMS} for a more detailed discussion. The fixed point set of this involution intersects each smooth fiber of $f$ in $8$ points. Identifying this smooth fibre with a quotient of the unit cube with opposite faces identified, this fixed point set coincides with the set of half integral points in the unit cube. 

Note that, in an affine neighbourhood of a vertex $v$, the strata of $\partial \Delta_{\PP^4}$ containing $v$ form a fan isomorphic to the toric fan of $\PP^3$. We identify this neighbourhood of $v$ with a domain in $\RR^3$ (with its standard integral affine structure) and let $\{e_i : i \in \{1,2,3\}\}$ denote the standard basis. We orient $\RR^3$ so that $(e_1,e_2,e_3)$ is a positively oriented basis. We identify the ray generators of the edges of $\partial \Delta_{\PP^4}$ which meet this neighbourhood with the vectors
\[
\{ e_1,e_2,e_3,-e_1-e_2-e_3\}.
\]
Writing $d_i := e_i$ for $i \in \{1,2,3\}$ and $d_4 := -e_1-e_2-e_3$, let $\tau_i$ denote the edge of $\Delta_{\PP^4}$ which contains $v$ and has tangent direction $d_i$ at $v$ for all $i \in \{1,\ldots,4\}$. Moreover, let $\sigma_{ij}$ denote the face of $\Delta_{\PP^4}$ containing the edges $\tau_i$ and $\tau_j$, for all pairs $i$,~$j \in \{1,\ldots,4\}$ such that $i \neq j$.

We now define loops $\gamma_{ij,k}$, for $i$,~$j \in \{1,\ldots,4\}$, and $k \in \{i,j\}$. These loops are based at $v$ and trace around a segment of the discriminant locus $\Delta$, as shown in Figure~\ref{fig:loops}, which is contained in $\sigma_{ij}$ and intersects edge $\tau_k$. We denote by $n_{ij}$ the primitive integral normal vector to $\sigma_{ij}$, such that the tuple $( d_i,d_j,n_{ij} )$ forms an ordered positive basis for $\RR^3$. We orient the loops $\gamma_{ij,k}$ by requiring that the tangent vector of $\gamma_{ij,k}$ at the unique point (other than $v$) at which the image of $\gamma_{ij,k}$ intersects $\sigma_{ij}$ pairs positively with $n_{ij}$. 

As described above, the eight real points on each fiber $b \in B_0$ which are invariant under the involution are precisely the $2$-torsion points of the torus $\Hom(\Lambda_b,U(1))$, recalling that $\Lambda$ denotes the sheaf of integral tangent vectors in $TB_0$. Identifying $\Hom(\Lambda_b,U(1))$ with $\RR^3/\ZZ^3$, these eight points are identified with the following eight vectors:
\begin{align*}
u_0 & = (0, 0, 0),\quad 
u_1  = \frac{1}{2}(0, 0, 1),\quad
u_2  = \frac{1}{2}(1, 0, 1 ),\quad
u_3  = \frac{1}{2}(1, 0, 0),\\
u_4 & = \frac{1}{2}(1, 1, 0),\quad
u_5  = \frac{1}{2}(0, 1, 0),\quad
u_6  = \frac{1}{2}(0, 1, 1),\quad
u_7  =\frac{1}{2}(1, 1, 1).
\end{align*}
The monodromy action $T'_{ij,k}$ around each loop $\gamma_{ij,k}$
is analysed in detail for the fibration on the quintic threefold $f\colon X \to B$ in \cite[Appendix~A]{AP19}.
\begin{figure}
	\includegraphics{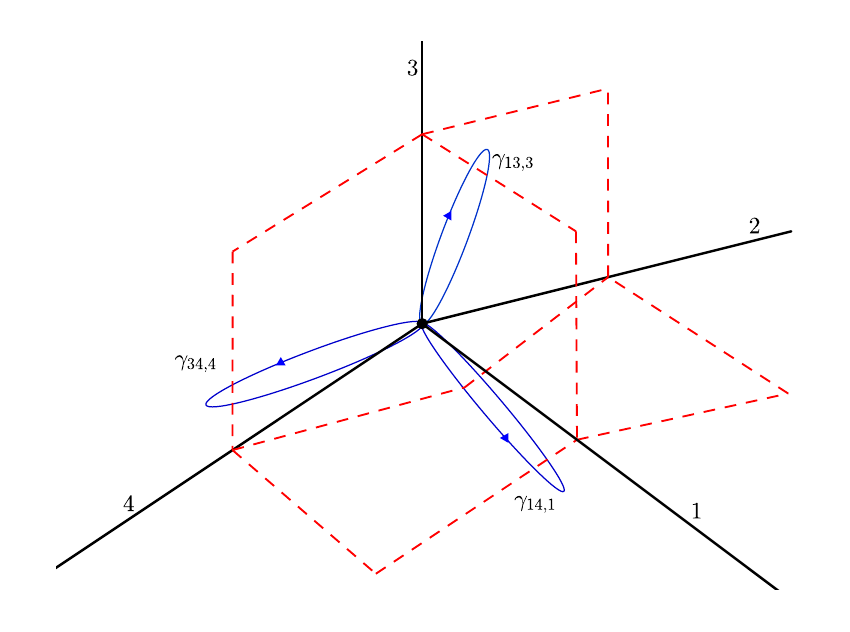}
	\caption{Examples of loops $\gamma_{ij,k}$.}
	\label{fig:loops}
\end{figure}
Recall that the mirror to the quintic admits a fibration $\breve{f} \colon \breve{X} \to B$, where the affine monodromy around the discriminant locus on $B$ is given by taking the inverse transpose of the affine monodromy on the base of the fibration for $f \colon X \to B$. Considering the action of $T_{ij,k}=({T'_{ij,k}})^{-t}$ on $\breve{L}_{\RR} \subset \breve{X}$ which permutes the torsion points $u_i$, for $i=\{1,\ldots,7\}$, we obtain the following set of double transpositions, similar to those described in \cite[Appendix~A]{AP19}:
\begin{align}
\label{Eq:monodromy for the mirror quintic}
T_{12,1}:(12)(67),  \quad & T_{12,2}:(16)(27) \quad T_{13,1}:(45)(67), \quad T_{13,3}:(47)(56) \\ 
\nonumber
T_{14,1}:(12)(45),  \quad & T_{14,4}:(14)(25) \quad T_{23,2}:(34)(27), \quad T_{23,3}:(23)(47) \\ 
\nonumber
T_{24,2}:(16)(34),  \quad & T_{24,4}:(14)(36) \quad T_{34,3}:(23)(56), \quad T_{34,4}:(25)(36).
\end{align}
The $2$-torsion points exchanged under these permutations are illustrated in Figure~\ref{fig:monodromy_cube}, in which the torsion point $u_i$ corresponds to the vertex with label $i$. Figure~\ref{fig:monodromy_cube} displays, from left to right:
\begin{enumerate}
	\item An example of the orbits of the $\ZZ_2$ action induced by monodromy around a single edge of $\Delta$.
	\item An example of the orbits of the $\ZZ^3_2$ action induced by monodromy around three different edges of $\Delta$ adjacent to a negative vertex.
	\item An example of the orbits of the $\ZZ^3_2$ action induced by monodromy around three different edges of $\Delta$ adjacent to a positive vertex.
\end{enumerate} 
Recall that the monodromy around negative vertices can be deduced from the monodromy around positive vertices by the relations in Theorem \ref{Thm:Marks}. Observe that $u_0$ remains invariant under the action of all of the monodromy matrices $T_{ij,k}$. That is, $u_0$ defines a section of $\breve{f} \colon X \to B$ and hence there is a connected component of the real Lagrangian $\breve{L}_\RR \subset \breve{X}$ homeomorphic to $S^3$. 
\begin{lemma}[{cf.\@\cite[Corollary~$1$]{CBM}}]
\label{lem:connected_cmpts}
The real locus $\breve{L}_\RR$ in the mirror to the quintic threefold consists of two connected components:
\[
\breve{L}_\RR = \breve{\shL}_\RR \textstyle \coprod S^3,
\]
where $\breve{\shL}_\RR$ is a $7$-to-$1$ cover of $B$ branched along the discriminant locus $\Delta \subset B$.
\end{lemma} 

\begin{figure}
	\makebox[\textwidth][c]{\includegraphics[scale=2]{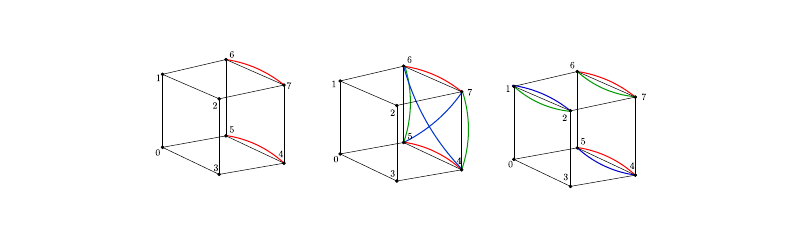}}
	\caption{The action of monodromy around a single branch, and the branches adjacent to a negative and positive vertex of $\Delta$ respectively}
	\label{fig:monodromy_cube}
\end{figure}
In the remaining part of this article we let
\[
\breve{f}_{\RR}\colon \breve{L}_{\RR} \to B \,\ \,\ \mathrm{and} \,\ \,\  \breve{\pi}\colon \breve{\mathcal{L}}_{\RR} \to B
\]
denote the restriction of $f \colon \breve{X} \to B$ to the real Lagrangian $\breve{L}_{\RR}$, and to the connected component $\breve{\mathcal{L}}_{\RR} \subset \breve{L}_{\RR}$, respectively.

\begin{proposition}
\label{pro:orientable}
	The real Lagrangian $\breve{L}_\RR \subset \breve{X}$ is the disjoint union of two orientable $3$-manifolds.
\end{proposition}
\begin{proof}
	Since one component of $\breve{L}_\RR$ is homeomorphic to $S^3$, it suffices to show that $\breve{\shL}_\RR$ is orientable. Consider a thickening of $\Delta \subset B = \partial \Delta_{\PP^4}$ obtained as follows. 
	\begin{enumerate}
		\item For each vertex $v$ of $\Delta$ fix a closed neighbourhood $W_v$ of $v$, homeomorphic to a $3$-ball.
		\item For each edge $e$ of $\Delta$, with vertices $v_1$ and $v_2$, we fix a small cylinder $W_e$ (homeomorphic to a $3$-ball) containing the intersection of $e$ with the complement of $W_{v_1} \cup W_{v_2}$. We assume that $W_e \cap W_{v_i}$ is a disc intersecting $\Delta$ in a single point for each $i \in \{1,2\}$.
	\end{enumerate}
	Define $W_1$ to be the the $3$-manifold (with boundary) given by the union of $\bigcup_v W_v$ and $\bigcup_e W_e$,
	where the unions are taken over the set of vertices and edges of $\Delta$ respectively. Note that, by construction, $W_1$ is a boundary connect sum of a disjoint set of $3$-balls (we refer to \cite[Remark~$1.3.3$]{GS99} for the definition of the boundary connect sum). Hence, by definition, it is a handlebody. Since $W_1$ as explained in the proof of Proposition \ref{pro:orientable} is obtained as a thickening of the discriminant locus $\Delta$, which is unknotted, the closure of the complement of $W_1$ in $B=S^3$, which we will denote by $W_2$, is also a handle body.
	Choosing an orientation of $B=S^3$ defines an orientation of $W_1$ and $W_2$. Moreover, by the analysis made in \cite[Appendix~A]{AP19} we deduce that $\breve{\pi}^{-1}(W_1)$ is a handlebody as well. Since $W_2$ has non-empty intersection with the discriminant locus in $B$, $\breve{\pi}^{-1}(W_2)$ is a disjoint union of homeomorphic copies of $W_2$, and hence is also a (disconnected) handlebody. In particular we note that
	\[
    \breve{\pi}^{-1}(W_1) \cap \breve{\pi}^{-1}(W_2)
    \]
    is an orientable surface. 
	Applying the Mayer--Vietoris sequence to the decomposition $\breve{\shL}_\RR \subset \breve{\pi}^{-1}(W_1) \cup \breve{\pi}^{-1}(W_2)$, we observe that
	\[
	H_3(\breve{\shL}_\RR) \cong H_2(\breve{\pi}^{-1}(W_1) \cap \breve{\pi}^{-1}(W_2)) \cong \ZZ,
	\]
	and hence $\breve{L}_\RR$ is an orientable manifold.
\end{proof}

It is a well-known fact in topology that every manifold admits a \emph{handle decomposition} and that in the three-dimensional case this is nothing but a \emph{Heegaard splitting}, see \cite[Chapter~$4.3$]{GS99}. By definition, a Heegaard splitting is a decomposition of a $3$-manifold into a pair of handlebodies glued along their boundaries. In particular, writing $\breve{\shL}_\RR = \breve{\pi}^{-1}(W_1) \cup \breve{\pi}^{-1}(W_2)$ determines a Heegaard splitting of $\breve{\shL}_\RR$. This provides a general approach to determining topological invariants of $\breve{\shL}_\RR$, which we exploit in \S\ref{sec:heegaard_splitting} to verify computations made in the proof of Theorem~\ref{thm:cech_to_derived}.

\begin{remark}
\label{rem:assumption}
    Note that Proposition~\ref{pro:orientable} applies in a much more general setting, to any integral affine structure on $S^3$ with simple singularities such that the complement of a thickening of $\Delta$ is a handlebody. In particular, this holds for every integral affine structure constructed in \cite{HZci,GrossBB} associated to a toric Calabi--Yau hypersurface.  
\end{remark}
\section{A \u{C}ech cover for $\breve{L}_{\RR}$}
\label{sec:topology_real_locus}
We describe an open cover 
\[
\U=\{\U_{\sigma_j^d}\}_{j \in J_d}
\] 
of $B$, where $J_d$ indexes $d$-dimensional faces of $\Delta_{\PP^4}$ and $\sigma_j^d$ denotes the $j$th $d$-dimensional face of $\Delta_{\PP^4}$. We will discuss the intersections of open sets in $\U$, and use this data in the following sections to compute the \u{C}ech cohomology groups of the real Lagrangian $\breve{L}_{\RR}$ described in \S\ref{sec:topology_real_locus}. Though we focus our attention to the mirror for the quintic threefold, the analysis in this section can be carried out in more general contexts, for instance when studying fibrations of Calabi--Yau hypersurfaces in smooth toric varieties, as well as more general Calabi--Yau hypersurfaces, as in Haase--Zharkov~\cite{HZci}.

\begin{construction}
	\label{def:coarse_cover}
	Let $\breve{f}\colon \breve{X} \to B$ be the torus fibration on the mirror to the quintic, as in \S\ref{sec:integral_affine_manifolds}.
	Recall that the polyhedral complex $B=\partial \Delta_{\PP^4}$ consists of five $3$-dimensional cells, given by five tetrahedra, which intersects each of the other four along a common two-dimensional face, ten $2$-cells given by the two-faces of the tetrahedra, ten $1$-cells given by edges, and $5$-vertices. We construct the open cover $\U$ as follows.
	\begin{itemize}
		\item[1)] First construct an open cover $\tilde{\U}=\{  \tilde{\U}_{\sigma_j^d} \}_{j\in J_d}$ from $\partial \Delta_{\PP^4}$, indexed by $d$-dimensional faces $\sigma_j^d$ of $\partial \Delta_{\PP^4}$ for $j\in J_d$, by fixing a neighbourhood $\tilde{\U}_{\sigma_j^d}$ of $\sigma_j^d$ homeomorphic to a $3$-ball which retracts to $\sigma_j^d$ for each $d \in \{0,1,2,3\}$ and $j \in J_d$.
		\item[2)] Replace each $\tilde{\U}_{\sigma_j^1} \in \tilde{U}$ by 
		\[
		\tilde{\U}_{\sigma_{j_1}^0} \cup \tilde{\U}_{\sigma_j^1} \cup \tilde{\U}_{\sigma_{j_2}^0}
		\] 
		where by $\sigma_{j_2}^0$ and $\sigma_{j_1}^0$ we denote the vertices adjacent to $\sigma_j^1$.
		\item[3)] Remove each open set $\tilde{\U}_{\sigma^0}$ from $\tilde{\U}$.
	\end{itemize}
	We denote the resulting open cover, obtained from $\tilde{\U}$ by the above steps, as $\U$. Shrinking the open sets $\tilde{\U}_{\sigma^d_j}$ as necessary, we assume that
	\[
	\U_{\sigma_i^2} \cap  \U_{\sigma_j^2} = \emptyset \,\ \mathrm{and} \,\ \U_{\sigma_i^3} \cap  \U_{\sigma_j^3} = \emptyset
	\]
	for any $i \neq j$. Moreover, we assume that
	\[
	\U_{\sigma_i^2} \cap \U_{\sigma_j^1} \neq \emptyset
	\]
	if and only if $\sigma_j^1$ is an edge of $\sigma_i^2$.
\end{construction}

In the remaining part of this section we will study the topology of the open sets $\breve{\pi}^{-1}(U) \subset \breve{\shL}_\RR$ for elements $U \in \U$. 

\begin{lemma}
\label{lem:3dfaces}
Let $\sigma^3$ be one of the $3$-cells of $\Delta_{\PP^4}$. Then 
$\breve{\pi}^{-1}(\U_{\sigma^3})$ is homeomorphic to the disjoint union of seven three-dimensional balls, and so in particular 
\[ H^0(\breve{\pi}^{-1}(\U_{\sigma^3}),\ZZ)=\ZZ^7 \]
and $H^j(\breve{\pi}^{-1}(\U_{\sigma^3}),\ZZ)=0$ for $j \geq 1$.
\end{lemma}

\begin{proof}
The map $\breve{\pi}$ is a $7$-to-$1$ branched cover over $B=\partial \Delta_{\PP^4}$, branched along the discriminant locus, which is contained in the $2$-cells of $\partial \Delta_{\PP^4}$.
As the discriminant locus does not intersect 
$\U_{\sigma^3}$, the restriction of $\breve{\pi}$ over $\U_{\sigma^3}$ is a trivial $7$-to-$1$ cover and so $\breve{\pi}^{-1}(\U_{\sigma^3})$ is homeomorphic to the disjoint union of seven three-dimensional balls. 
\end{proof}

\begin{figure}
	\includegraphics*{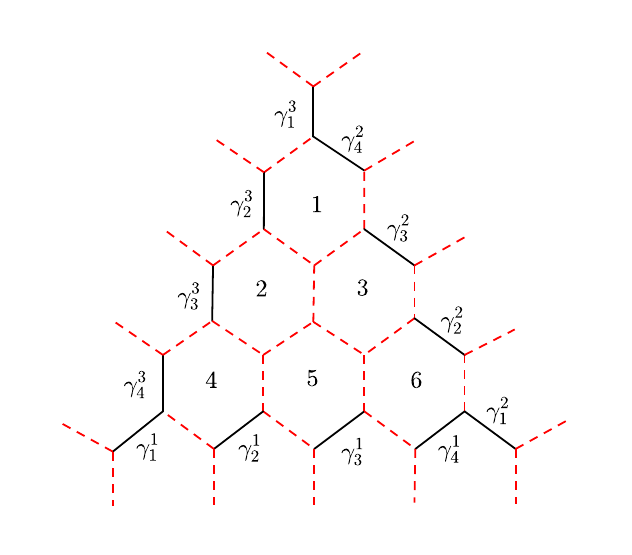}
	\caption{The hexagonal regions $H_i$ contained in a two-dimensional face of $\partial \Delta_{\PP^4}$ labelled by $i$, for $i \in \{1,\ldots , 6\}$.}
	\label{fig:generating_set}
\end{figure}

\begin{lemma}
	\label{lem:hexagonal_regions}
	Let $\sigma^2 \subset \partial \Delta_{\PP^4}$ be one of the $2$-cells of $\Delta_{\PP^4}$. These contain $6$ closed hexagonal regions $H_1,\ldots,H_6$, see Figure \ref{fig:generating_set}. For all $i \in \{1,\ldots,6\}$, the space $\breve{\pi}^{-1}(H_i)$ is the disjoint union of three discs and a connected CW complex $Y_i$. Moreover, we have that
	\[ H_0(Y_i,\ZZ)=\ZZ^4 \,,\,\,\,\,
	H_1(Y_i,\ZZ) \cong \ZZ^3 \oplus \ZZ_2 \,\ \mathrm{and} \,\ H_2(Y_i,\ZZ)= \{0\}
	\]
	for $i\in \{ 1,\ldots,6 \}$.
\end{lemma}

\begin{proof}
	From the description of the monodromy action on fibres of $\breve{\pi}$ given in \S\ref{sec:real_locus}, we observe that three points (labelled $1$, $2$, and $3$ in Figure~\ref{fig:monodromy_cube}) are monodromy invariant around loops around any branch of $\Delta$ which is contained in $\sigma$. That is, $\breve{\pi}^{-1}(H_i)$ consists of $4$ connected components, three of which define sections of $\breve{\pi}$ over $H_i$ and are hence homeomorphic to discs. We let $Y_i$ denote the remaining connected component.
	Moreover, we let $V_1$ be a neighbourhood of $\breve{\pi}^{-1}(\partial H_i) \cap Y_i$ which retracts onto $\breve{\pi}^{-1}(\partial H_i) \cap Y_i$, and set $V_2 := \breve{\pi}^{-1}(\Int(H_i)) \cap Y_i$. The space $V_1$ is homotopy equivalent to the union of six circles as shown in Figure~\ref{fig:attaching}, and we note that $V_1$ is homotopy equivalent to the wedge union of seven circles. Moreover, $V_2$ is homeomorphic to the disjoint union of four discs (recalling that, away from $\Delta$, $Y_i$ is a $4$-to-$1$ cover of $H_i$) and $V_1 \cap V_2$ is the disjoint union of four annuli, each contained in a unique connected component of $V_2$.
	
    Since the map $H_0(V_1\cap V_2,\ZZ) \to H_0(V_1,\ZZ) \oplus H_0(V_2,\ZZ)$ is injective, part of the Mayer--Vietoris sequence associated to the decomposition $Y_i = V_1 \cup V_2$ takes the following form. 
	\begin{equation}
	\label{eq:mayer_vietoris}
	   H_1(V_1\cap V_2,\ZZ) \to H_1(V_1,\ZZ) \oplus H_1(V_2,\ZZ) \to H_1(Y_i,\ZZ) \to 0.
	\end{equation}
	Applying our descriptions of the spaces $V_1$, $V_2$, and $V_1\cap V_2$, the sequence \eqref{eq:mayer_vietoris} has the form
	\[
	\ZZ^4 \stackrel{A}\lra \ZZ^7 \to H_1(Y_i,\ZZ) \to 0.
	\]
	We fix a basis of $H_1(V_1\cap V_2,\ZZ) \cong \ZZ^4$ by orienting the four sheets of the covering $V_1\cap V_2 \to \breve{\pi}(V_1 \cap V_2)$ clockwise. We also fix a basis of $H_1(V_1,\ZZ) \cong \ZZ^7$ by requiring that the first basis element corresponds to the cycle defined by the arcs in Figure~\ref{fig:attaching} with label $1$, oriented clockwise. The remaining six elements are chosen to be the homology classes of the six circles shown in Figure~\ref{fig:attaching}, with the indicated orientation, which map to edges of the hexagon $H_i$. Note that these conventions require that the plane containing the hexagon (and hence $\sigma^2$) is itself oriented. We achieve this by recalling that the vertices of $\Delta_{\PP^4}$ are labelled by elements of $\{1,\ldots, 5\}$, and requiring that the vertices of $\sigma^2$ appear in anti-clockwise order in the plane. Fixing these bases we have that
	\[
	A =
	\begin{pmatrix}
	1 & 1 & 1 & 1 \\
	0 & 1 & 1 & 0 \\
	0 & 1 & 0 & 1 \\
	0 & 0 & 1 & 1 \\
	0 & 1 & 1 & 0 \\
	0 & 1 & 0 & 1 \\
	0 & 0 & 1 & 1
	\end{pmatrix}.
	\]
	Noting that the first standard basis vector is in the image of $A$, we can discard it and present $H_1(Y_i,\ZZ)$ as a quotient of $\ZZ^6$. This quotient is generated by the standard basis elements $v_1,\ldots,v_6$ of $\ZZ^6$, subject to the relations $v_i+v_{i+1}+v_{i+3}+v_{i+4} = 0$ for $i \in \{1,2,3\}$, where the addition of indices is interpreted cyclically (in particular, the third relation is $v_3+v_4+v_6+v_1$). This quotient is isomorphic to $\ZZ^3 \oplus \ZZ_2$, and the elements $v_i$, $i \in \{1,2,3\}$ generate the torsion-free group, while $v_1+v_4$ generates the torsion subgroup. Finally, we note that, since $A$ is injective, and $H_2(V_1,\ZZ) = H_2(V_2,\ZZ) = \{0\}$, the group $H_2(Y_i,\ZZ)$ is trivial.
\end{proof}

\begin{lemma}
	\label{lem:preimage_face}
	Let $\sigma^2$ be a two-dimensional face of $\partial \Delta_{\PP^4}$, then 
	\begin{align*}
		H^0(\breve{\pi}^{-1}(\U_{\sigma^2}),\ZZ) \cong \ZZ^{4} 
		\,,\,\,\,\,
	H^1(\breve{\pi}^{-1}(\U_{\sigma^2}),\ZZ) \cong \ZZ^{12} \,,\,\,\,\,
	H^2(\breve{\pi}^{-1}(\U_{\sigma^2}),\ZZ) \cong \ZZ_2^6 \,,
	\end{align*}
and 
\begin{align*}
	H^0(\breve{\pi}^{-1}(\U_{\sigma^2}),\ZZ_2) \cong \ZZ_2^{4} 
		\,,\,\,\,\,
	H^1(\breve{\pi}^{-1}(\U_{\sigma^2}),\ZZ_2) \cong \ZZ_2^{18} 
	\,,\,\,\,\,
	H^2(\breve{\pi}^{-1}(\U_{\sigma^2}),\ZZ_2) \cong \ZZ_2^6 \,.
	\end{align*}
\end{lemma}
\begin{proof}
	
	
	We first note that $\breve{\pi}^{-1}(\U_{\sigma^2})$ retracts onto $\breve{\pi}^{-1}(\sigma^2)$. We observe that the space $\breve{\pi}^{-1}(\sigma^2)$ contains four connected components, three of which are homeomorphic to discs (sections of the restriction of $\breve{\pi}$ to $\sigma^2$). The remaining component retracts onto the wedge union of $\bigcup^6_{i=1}Y_i$ and three circles. These three circles are contained in the pre-images of the segments labelled $\gamma_1^i$ for $i \in \{1,2,3\}$ in Figure~\ref{fig:generating_set}. Thus, setting 
	\[
	\bar{Y}_k := \bigcup^k_{i=1}Y_i,
	\]
	and $Y := \bar{Y}_6$, we conclude that 
	\[
	H^1(\breve{\pi}^{-1}(\U_{\sigma^2}),\ZZ) \cong H^1(Y,\ZZ) \oplus \ZZ^3 \,,\,\,\,\,
	H^2(\breve{\pi}^{-1}(\U_{\sigma^2}),\ZZ) \cong H^2(Y,\ZZ)\] and
	\[
	H^1(\breve{\pi}^{-1}(\U_{\sigma^2}),\ZZ_2) \cong H^1(Y,\ZZ_2) \oplus \ZZ_2^3 \,,\,\,\,\,
	H^2(\breve{\pi}^{-1}(\U_{\sigma^2}),\ZZ_2) \cong H^2(Y,\ZZ_2)\,.\]
	\begin{figure}
		\includegraphics*{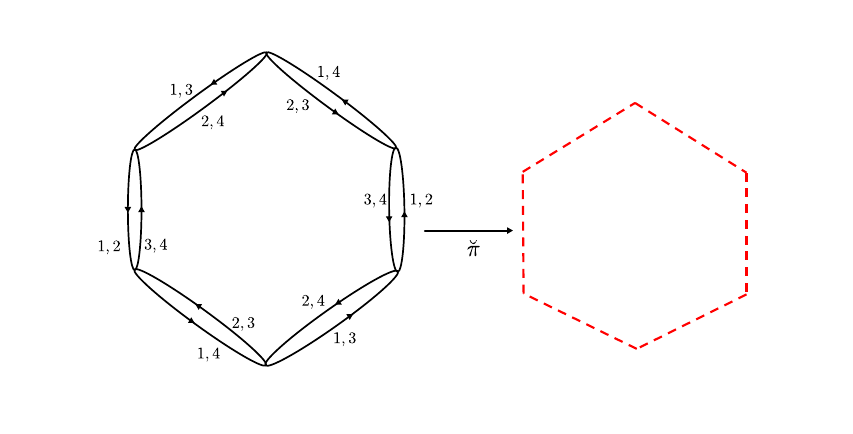}
		\caption{The restriction of $\breve{\pi}$ to the pre-image of a hexagonal region $H_i$.}
		\label{fig:attaching}
	\end{figure}

	Fixing a value of $k \in \{1,\ldots, 5\}$, and using that 
	$H_2(Y_{k+1},\ZZ)=0$ by Lemma~\ref{lem:hexagonal_regions}, 
	part of the Mayer--Vietoris sequence associated to the decomposition $\bar{Y}_{k+1} = Y_{k+1} \cup \bar{Y}_k$ has the form
	\[
	0 \to H_2(\bar{Y}_{k},\ZZ) \to
	H_2(\bar{Y}_{k+1},\ZZ)
	\to \ZZ^{l_k} \stackrel{M}{\lra} H_1(Y_{k+1},\ZZ) \oplus H_1(\bar{Y}_k,\ZZ) \to H_1(\bar{Y}_{k+1},\ZZ) \to 0,
	\]
	where $l_k \in \{1,2,3\}$ denotes the number of circles in the intersection $Y_{k+1} \cap \bar{Y}_k$.
	
	By Lemma~\ref{lem:hexagonal_regions}, we have $H_1(Y_k,\ZZ) \cong \ZZ^3\oplus \ZZ_2$  for all $k \in \{1,\ldots, 6\}$. The component $\ZZ^{l_k} \to H_1(Y_{k+1},\ZZ)$ of the map $M$ is injective, and hence $M$ is itself injective. It follows that  $H_2(\bar{Y}_{k},\ZZ) \simeq  H_2(\bar{Y}_{k+1},\ZZ)$. Using the base case $H_2(Y_1,\ZZ)=0$ given by  Lemma~\ref{lem:hexagonal_regions}, we deduce by induction that $H_2(\bar{Y}_{k+1},\ZZ)=0$ for every $k \in \{1,\dots,5\}$, and in particular 
	\[ H_2(Y,\ZZ)=0 \,.\]
	Moreover, the image of $M$ is a saturated subgroup of the torsion free part of $H_1(Y_{k+1},\ZZ) \oplus H_1(\bar{Y}_k,\ZZ)$. Thus, assuming inductively that $H_1(\bar{Y}_{k},\ZZ) \cong \ZZ^{a_k} \oplus \ZZ_2^k$ for some $a_k \in \ZZ_{>0}$, we have that
	\[
	H_1(\bar{Y}_{k+1},\ZZ) \cong \ZZ^{a_k+3-l_k} \oplus \ZZ_2^{k+1}.
	\]
	This, together with the base case that $H_1(\bar{Y}_1,\ZZ) = H_1(Y_1,\ZZ) \cong \ZZ^3 \oplus \ZZ_2$ verifies that
	\[
	H_1(\bar{Y}_k,\ZZ) \cong \ZZ^{a_k} \oplus \ZZ_2^k
	\]
	for all $k \in \{1,\ldots,6\}$. Moreover, computing $a_k$ using the formula $a_{k+1} = a_k + 3 - l_k$, we deduce that 
	\[
	H_1(Y,\ZZ) = H_1(\bar{Y}_6,\ZZ) \cong \ZZ^9 \oplus \ZZ_2^6 \,.
	\]
	As $H_2(Y,\ZZ)=0$, we have using the universal coefficient theorem that
	\[ H^1(Y,\ZZ)=\mathrm{Hom}(H_1(Y,\ZZ),\ZZ)=\ZZ^9 \,,\,\,\,\, 
	H^2(Y,\ZZ)=\mathrm{Ext}^1(H_1(Y,\ZZ),\ZZ)=\ZZ_2^6 \,,\] 
	and 
	\[ H^1(Y,\ZZ_2)=\mathrm{Hom}(H_1(Y,\ZZ),\ZZ_2)=\ZZ_2^{15}\,,\,\,\,\, 
	H^2(Y,\ZZ_2)=\mathrm{Ext}^1(H_1(Y,\ZZ),\ZZ_2)=\ZZ_2^6 \,.\]
	
\end{proof}



\begin{lemma}
\label{lem:topology_for_edges}
For each edge $\sigma^1 \subset \partial \Delta_{\PP^4}$, there is a homeomorphism between $\breve{\pi}^{-1}(U_{\sigma^1})$ and the disjoint union of $3$ copies of an open $3$-manifold which retracts onto the wedge union of $4$ circles, and a $3$-dimensional ball.
In particular, we have 
\[ H_0(\breve{\pi}^{-1}(U_{\sigma^1}),\ZZ)=\ZZ^4 \,,\,\,\,\,H_1(\breve{\pi}^{-1}(U_{\sigma^1}),\ZZ)=\ZZ^{12} \,,\,\,\,\,
H_2(\breve{\pi}^{-1}(U_{\sigma^1}),\ZZ)=0 \,,\]
\[ H^0(\breve{\pi}^{-1}(U_{\sigma^1}),\ZZ)=\ZZ^4 \,,\,\,\,\,H^1(\breve{\pi}^{-1}(U_{\sigma^1}),\ZZ)=\ZZ^{12} \,,\,\,\,\,
H^2(\breve{\pi}^{-1}(U_{\sigma^1}),\ZZ)=0 \,,\]
and 
\[H^0(\breve{\pi}^{-1}(U_{\sigma^1}),\ZZ_2)=\ZZ_2^4 \,,\,\,\,\, H^1(\breve{\pi}^{-1}(U_{\sigma^1}),\ZZ_2)=\ZZ_2^{12} \,,\,\,\,\,
H^2(\breve{\pi}^{-1}(U_{\sigma^1}),\ZZ_2)=0 \,,\]
\end{lemma}
\begin{proof}
    We recall that each vertex of $\Delta$ which is contained in $\sigma^1$ is a \emph{positive vertex}. Fixing a base point $x \in \sigma^1 \setminus \Delta$, monodromy actions on $\breve{\pi}^{-1}(x)$ along loops passing around branches of $\Delta$ containing this vertex are shown in the central image of Figure~\ref{fig:monodromy_cube}. Hence one element of $\breve{\pi}^{-1}(x)$ (labelled $3$ in Figure~\ref{fig:monodromy_cube}) is contained in a section of $\breve{\pi}^{-1}(\U_{\sigma^1})$ over $\U_{\sigma^1}$. Moreover, the pairs of points labelled $\{1,2\}$, $\{4,5\}$, and $\{6,7\}$ lie in distinct connected components of $\breve{\pi}^{-1}(\U_{\sigma^1})$. Letting $x$ vary along $\sigma^1$, elements in these three pairs come together over the five points in the intersection $\Delta \cap \sigma^1$, and hence each of these components retracts onto the wedge union of $4$ circles.
\end{proof}

\begin{lemma}
\label{lem:topology_for_intersections}
Given a two-dimensional face $\sigma^2$ of $\Delta_{\PP^4}$, and an edge $\sigma^1$ of $\sigma^2$, we have that
\[
\breve{\pi}^{-1}(\U_{\sigma^1} \cap \U_{\sigma^2})
\]
contains five connected components. Three of these five components are homeomorphic to a $3$-ball, while the other two retract onto the wedge union of $4$ circles.
In particular, we have 
\[ H_1(\breve{\pi}^{-1}(\U_{\sigma^1} \cap \U_{\sigma^2}),\ZZ)=\ZZ^8 \,,\,\,\,\,
H_2(\breve{\pi}^{-1}(\U_{\sigma^1} \cap \U_{\sigma^2}),\ZZ)=0 \,,\]
\[ H^1(\breve{\pi}^{-1}(\U_{\sigma^1} \cap \U_{\sigma^2}),\ZZ)=\ZZ^8 \,,\,\,\,\,
H^2(\breve{\pi}^{-1}(\U_{\sigma^1} \cap \U_{\sigma^2}),\ZZ)=0 \,,\]
and 
\[ H^1(\breve{\pi}^{-1}(\U_{\sigma^1} \cap \U_{\sigma^2}),\ZZ_2)=\ZZ_2^8 \,,\,\,\,\,
H^2(\breve{\pi}^{-1}(\U_{\sigma^1} \cap \U_{\sigma^2}),\ZZ_2)=0 \,,\]
\end{lemma}
\begin{proof}
This result follows from the same analysis used to prove Lemma~\ref{lem:topology_for_edges}. Note that $\U_{\sigma^1} \cap \U_{\sigma^2}$ contains no vertices of $\Delta$, and monodromy actions along loops around a single segment of $\Delta$ are illustrated in the left-hand image of Figure~\ref{fig:monodromy_cube}.
\end{proof}



\section{The integral cohomology of $\breve{L}_\RR$}
\label{sec:betti_numbers}


The main result of this section concerns the computation of a \u{C}ech-to-derived spectral sequence for the sheaf $\breve{\pi}_\star\ZZ$. We recall that the \u{C}ech-to-derived spectral sequence for a sheaf $\shF$ on $B$ with open cover $\U$ has the form
\[
E_2^{p,q}= \breve{H}^p(\U,\shH^q(\cdot,\shF)) \implies H^{p+q}(B,\shF) \,,
\]
where $\shH^q(\cdot,\shF)$ is the presheaf 
$U \mapsto H^q(U,\shF)$ on $B$, and for every presheaf
$\shG$ on $B$, $\breve{H}^p(\U,\shG)$ denotes the 
p'th \u{C}ech cohomology group of $\shG$
with respect to the cover $\U$.
Analysing this sequence for $\shF=\breve{\pi}_\star\ZZ$,
where $\breve{\pi}$ is 
the branched covering $\breve{\pi} \colon \breve{\shL}_\RR \to B$,
we obtain the following result.
\begin{theorem}
	\label{thm:cech_to_derived}
	Let $\U$ be the open cover of $B$ defined in 
	Construction~\ref{def:coarse_cover}.
	The \u{C}ech-to-derived spectral sequence 
	\[
E_2^{p,q}= \breve{H}^p(\U,\shH^q(B,\breve{\pi}_\star\ZZ)) \implies H^{p+q}(B,\breve{\pi}_\star\ZZ) \,,
\]
	for the sheaf $\breve{\pi}_\star\ZZ$, has the $E_2$ page
	\[
	\xymatrix@R-2pc{
		\ZZ_2^{60} & & & \\
		0 & \ZZ_2^{36} \oplus \ZZ_4^6 \oplus \ZZ_8^4 \oplus \ZZ^2_{32} & & \\
		\ZZ & 0 & \ZZ_2 & \ZZ.
	}
	\]
	Moreover, this spectral sequence degenerates at the $E_2$ page.
\end{theorem}

\begin{proof}
	
	We need to compute seven \u{C}ech cohomology groups, as displayed in the statement of Theorem~\ref{thm:cech_to_derived}. To describe these cohomology groups we first define the presheaf
	\[
	\shH^j_\ZZ \colon U \mapsto H^j(\breve{\pi}^{-1}(U), \ZZ)
	\]
	for open sets $U \subset B$, $j \in \ZZ_{\geq 0}$. The non-zero \u{C}ech cohomology groups we need to compute are:
	\[
	\xymatrix@R-2pc{
		\breve{H}^0(\U,\shH^2_\ZZ) & & & \\
		\breve{H}^0(\U,\shH^1_\ZZ) & \breve{H}^1(\U,\shH^1_\ZZ) & & \\
		\breve{H}^0(\U,\shH^0_\ZZ) & \breve{H}^1(\U,\shH^0_\ZZ) & \breve{H}^2(\U,\shH^0_\ZZ) & \breve{H}^3(\U,\shH^0_\ZZ).
	}
	\]
	These \u{C}ech cohomology groups are obtained from associated \u{C}ech complexes (which form the $E_1$ page of this spectral sequence). We let
	\[
	\breve{C}^i(\U,\shH^j_\ZZ)
	\]
	denote the group of \u{C}ech $i$-cochains for the presheaf $\shH^j_\ZZ$ associated to the open cover $\U$.

    Recall from Construction~\ref{def:coarse_cover} that the open cover $\U$ consists of 
    \begin{enumerate}
        \item $10$ open sets $\U_{\sigma^1}$, indexed by the $1$-dimensional faces $\sigma^1$ of $B=\partial \Delta_{\PP^4}$.
        \item 10 open sets $\U_{\sigma^2}$, indexed by the $2$-dimensional faces $\sigma^2$ of $B=\partial \Delta_{\PP^4}$.
        \item 5 open sets $\U_{\sigma^3}$, indexed by the $3$-dimensional faces $\sigma^3$ of $B=\partial \Delta_{\PP^4}$.
    \end{enumerate}

	By Lemma~\ref{lem:topology_for_edges}, 
	we have, for every $\sigma^1$,
	\begin{equation}
	\label{eq:cohomology_edges}
	H^i(\breve{\pi}^{-1}(\U_{\sigma^1}),\ZZ) =
	\begin{cases}
	\ZZ^4 & i=0 \\
	 \ZZ^{12} & i = 1\\
	 0 & i = 2 \,.
	\end{cases}
	\end{equation}
	
	By Lemma~\ref{lem:preimage_face}, we have, for every 
	$\sigma^2$,
	\begin{equation}
	\label{eq:cohomology_faces}
	H^i(\breve{\pi}^{-1}(\U_{\sigma^2}),\ZZ) =
	\begin{cases}
	\ZZ^4 & i=0 \\
	 \ZZ^{12} & i = 1\\
	 \ZZ^6_2 & i = 2.
	\end{cases}
	\end{equation}
	
	By Lemma \ref{lem:3dfaces}, we have, for every $\sigma^3$,  
	\begin{equation}
	\label{eq:cohomology_3faces}
	H^i(\breve{\pi}^{-1}(\U_{\sigma^3}),\ZZ) =
	\begin{cases}
	\ZZ^7 & i=0 \\
	 0 & i \geq 1\,.
	\end{cases}
	\end{equation}

	Since the $10$ open sets $\U_{\sigma^2}$ are the only elements of $\U$ whose pre-image under $\breve{\pi}$ has non-zero second cohomology, we have 
	\[
	\breve{H}^0(\U,\shH^2_\ZZ) \cong \breve{C}^0(\U,\shH^2_\ZZ) \cong \ZZ_2^{10 \times 6} = \ZZ^{60} \,.
	\]

    The cohomology groups $\breve{H}^i(\U,\shH^1_\ZZ)$ for $i \in \{1,2\}$ are determined by the two-term \u{C}ech complex
    \[
    \breve{C}^0(\U,\shH^1_\ZZ) \stackrel{\delta}\lra \breve{C}^1(\U,\shH^1_\ZZ).
    \]
    The group $\breve{C}^0(\U,\shH^1_\ZZ)$ is the sum of first cohomology groups of the preimages by $\breve{\pi}$ of the 
    $10$ open sets $U_{\sigma^1}$ and of the $10$ open sets 
    $U_{\sigma^2}$.
    Using  \eqref{eq:cohomology_edges} and \eqref{eq:cohomology_faces}, we obtain
    \[
    \breve{C}^0(\U,\shH^1_\ZZ) \cong \ZZ^{12\times 10 + 12\times 10} = \ZZ^{240}.
    \]
    The group $\breve{C}^1(\U,\shH^1_\ZZ)$ is the sum of the cohomology groups
    \[
    H^1(\breve{\pi}^{-1}(\U_{\sigma^1_j} \cap \U_{\sigma^2_k}),\ZZ)
    \]
    where $\sigma^1_j$ is an edge of $\sigma^2_k$. By Lemma~\ref{lem:topology_for_intersections} these groups are all isomorphic to $\ZZ^8$. Since there are $30$ such intersections, we have that
    \[
    \breve{C}^1(\U,\shH^1_\ZZ) \cong \ZZ^{30\times 8} = \ZZ^{240},
    \]
    and hence the map $\delta$ is determined by a $240 \times 240$ integer matrix $[\delta]$. Magma~\cite{Magma} source code for this construction is included in the supplementary material~\cite{Supplementary}. The computation of $\breve{H}^i(\U,\shH^1_\ZZ)$ for $i \in \{0,1\}$ follows from the Smith normal form of $[\delta]$.

	Finally, we consider the bottom row of the spectral sequence appearing in Theorem~\ref{thm:cech_to_derived}. We compute the ranks of the groups in the corresponding \u{C}ech complex by counting connected components lying over each open set. We obtain the sequence
	\[
	0 \lra \ZZ^{115} \stackrel{\delta_0}\lra \ZZ^{710} \stackrel{\delta_1}\lra \ZZ^{1190} \stackrel{\delta_2}\lra \ZZ^{595} \lra 0,
	\]
	which we verify has Euler characteristic zero. 
	With more details, there are $5$ (resp. $10$ and $10$) open subsets of the form 
	$\U_{\sigma^3}$ (resp.\ $\U_{\sigma^2}$ and $\U_{\sigma^3}$) over which lie $7$ (resp.\ $4$ and $4$) connected components, and 
	$5 \times 7+10 \times 4+10 \times 4=115$.
	There are $60$ (resp. $60$, $30$ and $20$) triple intersections  of types
	$\U_{\sigma^3} \cap \U_{\sigma^2} \cap \U_{\sigma_1}$ (resp.\ 
		$\U_{\sigma^3} \cap \U_{\sigma^1} \cap \U_{\sigma_1}$ , $\U_{\sigma^2} \cap \U_{\sigma^1} \cap \U_{\sigma_1}$  and 	$\U_{\sigma^1} \cap \U_{\sigma^1} \cap \U_{\sigma_1}$ ) over which lie $7$ (resp.\ $7$, $7$ and $7$) connected components, and 
	$60 \times 7+60 \times 7+30 \times 7+20 \times 7=1190$.
		There are $20$ (resp. $30$, $30$ and $30$) double intersections  of types
	$\U_{\sigma^3} \cap \U_{\sigma^2}$ (resp.\ 
		$\U_{\sigma^3} \cap \U_{\sigma^1}$, $\U_{\sigma^2} \cap \U_{\sigma^1}$ and 	$\U_{\sigma^1} \cap \U_{\sigma^1}$) over which lie $7$ (resp.\ $7$, $5$ and $7$) connected components, and 
	$20 \times 7+30 \times 7+30 \times 5+30 \times 7=710$.
	There are $60$ (resp.\ $20$ and $5$) quadruple intersections of the form $\U_{\sigma^3} \cap\U_{\sigma^2} \cap \U_{\sigma^1}  \cap \U_{\sigma^1}$
	(resp.\ $\U_{\sigma^3} \cap\U_{\sigma^1} \cap \U_{\sigma^1}  \cap \U_{\sigma^1}$ and $\U_{\sigma^1} \cap\U_{\sigma^1} \cap \U_{\sigma^1}  \cap \U_{\sigma^1}$)
	 over which lie $7$ (resp.\ $7$ and $7$) connected components,
and $60 \times 7+20 \times 7+5 \times 7=595$.
	The maps $\delta_1$ and $\delta_2$ are computed using Magma and source code for this computation is included in the supplementary material. In fact, since the cochain groups are torsion free,  the Smith normal form of the \u{C}ech differentials immediately determines the analogous result using the coefficient ring $\ZZ_2$.

	We note that degeneration of the spectral sequence at the $E_2$ page follows immediately from the fact that the only homomorphism from a torsion group to a free group is the zero map.
\end{proof}

While Theorem~\ref{thm:cech_to_derived} does not determine the integral cohomology of $\breve{\shL}_\RR$ up to isomorphism we can deduce the following properties.

\begin{corollary}
\label{Cor: integer cohomology}
We can describe the cohomology groups of $\breve{\shL}_\RR$ as follows.
\begin{enumerate}
    \item $H^0(\breve{\shL}_\RR,\ZZ) \cong H^3(\breve{\shL}_\RR,\ZZ) \cong \ZZ$.
    \item $H^1(\breve{\shL}_\RR,\ZZ) \cong 0$.
    \item $H^2(\breve{\shL}_\RR,\ZZ) \cong T$, where $T$ is a $2$-primary finite abelian group such that every element has order $\leq 2^7$.
\end{enumerate}
\end{corollary}
\begin{proof}
   Since $\breve{\shL}_\RR$ is connected and oriented we have that $H^0(\breve{\shL}_\RR,\ZZ) \cong H^3(\breve{\shL}_\RR,\ZZ) \cong \ZZ$. The remaining items follow immediately from the description of the $E_2$ page of the spectral sequence which appears in Theorem~\ref{thm:cech_to_derived}. In particular, the groups $E^{0,2}_2$ and $E^{1,1}_2$ are graded pieces of a filtration on $H^{p+q}(\breve{\shL}_\RR,\ZZ)$. Hence $H^2(\breve{\shL}_\RR,\ZZ)$ has a filtration with graded pieces equal to $\ZZ_2^{60}$, $\ZZ_2^{36} \oplus \ZZ_4^6\oplus\ZZ_8^4 \oplus \ZZ_{32}^2$, and $\ZZ_2$ respectively. We note that the maximal order of an element in such an extension is $128$.
  \end{proof}
We note that, since $H^2(\breve{\shL}_\RR,\ZZ_2) \cong \ZZ_2^{101}$, the filtration of $H^2(\breve{\shL}_\RR,\ZZ)$ described in the proof of Corollary~\ref{Cor: integer cohomology} does not split.
\section{The mod two cohomology of $\breve{L}_\RR$ and Hodge numbers of $\breve{X}$}
\label{sec:mod2_cohomology}

\subsection{Hodge numbers \`{a} la Batyrev}
\label{subsec:Batyrev Hodge}

It is shown in \cite{B} that one can construct a family of Calabi--Yau toric hypersurfaces from a given reflexive lattice polytope. Moreover, it is shown that the Hodge numbers of these hypersurfaces can be computed by enumerating integral points on the associated polytope.

\begin{example}
A toric degeneration of the quintic threefold is given by 
\begin{equation}
\label{Eq: quintic family}    
\shX=(x_0x_1x_2x_3x_4+tf_5) \subset \PP^4\times\AA^1_t    
\end{equation}
where $f_5$ is a degree $5$ homogeneous polynomial. A mirror $\breve{X}$ to the quintic threefold can be obtained via a crepant resolution of a quotient of the general fiber of the family \eqref{Eq: quintic family}, see, for example, \cite{Auroux} for further details. Moreover, there is a toric degeneration of the mirror such that the corresponding intersection complex is given by $B = \partial \Delta_{\PP^4}$, together with 
 the polyhedral decomposition of $\partial \Delta_{\PP^4}$ described in \S\ref{sec:integral_affine_manifolds}. We refer to \cite[Appendices A,B]{AP19} for a survey of this example. 
\end{example}


Hence, restricting our attention to the mirror $\breve{X}$ to the quintic threefold, where $B=\partial \Delta_{\PP^4}$ as discussed in \S\ref{sec:integral_affine_manifolds}, and using \cite[Theorem~$4.3.1$]{B}, we deduce that the Hodge number 
\[
h^{1,1}(\breve{X})=h^{2,1}(X)
\] 
can be computed in terms of the integral points on $\Delta_{\PP^4}$ from the equation
\begin{equation}
\label{Eq: Batyrev Hodge}    
 h^{2,1}(X)=l(\Delta_{\PP^4})-n-1-\sum_{\codim(\theta)=1}l^\star(\theta),
\end{equation}
where $n=4$ since $X \subset \PP^4$, and the other terms are defined as follows: $l(\Delta_{\PP^4})$ is the number of integral points in $\Delta_{\PP^4}$, the term $l^\star(\theta)$ denotes the number of integral points in the interior of a face $\theta$, and the sum is over all $\codim 1$ faces $\theta$ of $\Delta_{\PP^4}$, see \cite[Definition~$3.3.6$]{B}. It is shown in \cite[\S~$4.1.1$]{AM126} that $l(\Delta_{\PP^4})=126$. Note that there are $l^\star(\theta)=4$ integral points in the interior of each of the $5$ tetrahedra. Hence, we obtain
\[
h^{2,1}(X)=126-4-1-4\cdot5=101,
\]
as expected. In \S\ref{subsec: mod two} we give a correspondence between integral points on $\partial \Delta_{\PP^4}$ and cycles generating cohomology groups of the real Lagrangian $\breve{\mathcal{L}}_{\RR} \subset \breve{X}$, described in \S\ref{sec:real_locus}.

\subsection{The mod $2$ cohomology of $\breve{\mathcal{L}_{\RR}}$}
\label{subsec: mod two}

In \cite{AP19}, we computed the mod $2$ cohomology groups of Calabi--Yau threefolds which admit topological torus fibrations over integral affine manifolds with simple singularities, as discussed in \S\ref{sec:SYZ Mirrors}. In particular, we obtained the following result, see \cite[Example~$4.10$]{AP19}.

\begin{proposition}
Let $\breve{X}$ be the mirror to the quintic threefold, and $\breve{L}_\RR$ be the real Lagrangian obtained as in \S\ref{sec:real_locus}. Then, $\breve{L}_\RR$ is the disjoint union of the $3$-sphere with a $7$-to-$1$ branched cover $\breve{\mathcal{L}}_\RR$ over it \cite[Lemma~$3.2$]{AP19}, and the ranks of mod $2$ cohomology groups of $\breve{L}_\RR$ are
\begin{eqnarray}
\nonumber
h^0(\breve{L}_\RR,\ZZ_2) & = & h^3(\breve{L}_\RR,\ZZ_2) = 2.\\
\nonumber
h^1(\breve{L}_\RR,\ZZ_2) & = & h^2(\breve{L}_\RR,\ZZ_2) = 101.
\end{eqnarray}
\end{proposition}

In this section we further investigate the mod $2$ cohomology of $\breve{\shL}_\RR$, using a \u{C}ech-to-derived spectral sequence. In particular, we show how to identify a basis of graded pieces of a filtration in $H^1(\breve{\shL}_\RR,\ZZ_2) \cong \ZZ_2^{101}$ with sets of integral points in $\Delta_{\PP^4}$. This allows us to relate the mod $2$ cohomology groups of $\breve{\shL}_\RR$, to Hodge numbers of $\breve{X}$ using \cite{B}, as outlined in \S\ref{subsec:Batyrev Hodge}.

\begin{theorem}
	\label{thm:cech_to_derived_Z2}
		Let $\U$ be the open cover of $B$ defined in 
	Construction~\ref{def:coarse_cover}.
	The \u{C}ech-to-derived spectral sequence 
	\[
E_2^{p,q}= \breve{H}^p(\U,\shH^q(B,\breve{\pi}_\star\ZZ_2)) \implies H^{p+q}(B,\breve{\pi}_\star\ZZ_2) \,,
\]
	for the sheaf $\breve{\pi}_\star\ZZ_2$, has the $E_2$ page
	\begin{equation}
	\label{eq:E2_page_mod2}
	\xymatrix@R-2pc{
		\ZZ_2^{60} & & & \\
		\ZZ_2^{100} & \ZZ_2^{40} & & \\
		\ZZ_2 & \ZZ_2 & \ZZ_2 & \ZZ_2.
	}
	\end{equation}
	Moreover, this spectral sequence degenerates at the $E_2$ page.
\end{theorem}
\begin{proof}
	We need to compute seven \u{C}ech cohomology groups, the non-zero entries in each of the $E_2$ pages displayed in the statement of Theorem~\ref{thm:cech_to_derived_Z2}. To describe these cohomology groups we first define the pre-sheaf
	\[
	\shH^j_{\ZZ_2} \colon U \mapsto H^j(\breve{\pi}^{-1}(U), \ZZ_2)
	\]
	for open sets $U \subset B$, $j \in \ZZ_{\geq 0}$. The non-zero \u{C}ech cohomology groups we need to compute are:
	\[
	\xymatrix@R-2pc{
		\breve{H}^0(\U,\shH^2_{\ZZ_2}) & & & \\
		\breve{H}^0(\U,\shH^1_{\ZZ_2}) & \breve{H}^1(\U,\shH^1_{\ZZ_2}) & & \\
		\breve{H}^0(\U,\shH^0_{\ZZ_2}) & \breve{H}^1(\U,\shH^0_{\ZZ_2}) & \breve{H}^2(\U,\shH^0_{\ZZ_2}) & \breve{H}^3(\U,\shH^0_{\ZZ_2}).
	}
	\]
	
	  Recall from Construction~\ref{def:coarse_cover} that the open cover $\U$ consists of 
    \begin{enumerate}
        \item $10$ open sets $\U_{\sigma^1}$, indexed by the $1$-dimensional faces $\sigma^1$ of $B=\partial \Delta_{\PP^4}$.
        \item 10 open sets $\U_{\sigma^2}$, indexed by the $2$-dimensional faces $\sigma^2$ of $B=\partial \Delta_{\PP^4}$.
        \item 5 open sets $\U_{\sigma^3}$, indexed by the $3$-dimensional faces $\sigma^3$ of $B=\partial \Delta_{\PP^4}$.
    \end{enumerate}

	By Lemma~\ref{lem:topology_for_edges}, 
	we have, for every $\sigma^1$,
	\begin{equation}
	\label{eq:cohomology_edges_2}
	H^i(\breve{\pi}^{-1}(\U_{\sigma^1}),\ZZ_2) =
	\begin{cases}
	\ZZ_2^4 & i=0 \\
	 \ZZ_2^{12} & i = 1\\
	 0 & i = 2 \,.
	\end{cases}
	\end{equation}
	
	By Lemma~\ref{lem:preimage_face}, we have, for every 
	$\sigma^2$,
	\begin{equation}
	\label{eq:cohomology_faces_2}
	H^i(\breve{\pi}^{-1}(\U_{\sigma^2}),\ZZ_2) =
	\begin{cases}
	\ZZ_2^4 & i=0 \\
	 \ZZ_2^{18} & i = 1\\
	 \ZZ_2^6 & i = 2.
	\end{cases}
	\end{equation}
	
	By Lemma \ref{lem:3dfaces}, we have, for every $\sigma^3$,  
	\begin{equation}
	\label{eq:cohomology_3faces_2}
	H^i(\breve{\pi}^{-1}(\U_{\sigma^3}),\ZZ_2) =
	\begin{cases}
	\ZZ_2^7 & i=0 \\
	 0 & i \geq 1\,.
	\end{cases}
	\end{equation}

The only contribution to $\breve{H}^0(\U,\shH^2_{\ZZ_2})$
comes from the $10$ open sets $\U_{\sigma^2}$ via \eqref{eq:cohomology_faces_2} and so
	\[
	\breve{H}^0(\U,\shH^2_{\ZZ_2}) \cong \breve{C}^0(\U,\shH^2_{\ZZ_2}) \cong \ZZ_2^{10 \times 6}= \ZZ_2^{60} \,.
	\]
	We now compute the \u{C}ech differential 
	\[
	\delta \colon \breve{C}^0(\U,\shH^1_{\ZZ_2}) \to \breve{C}^1(\U,\shH^1_{\ZZ_2}).
	\]
	 The dimensions $\dim \breve{H}^i(\U,\shH^1_{\ZZ_2})$ are equal to $\dim \ker (\delta)$ and $\dim \coker (\delta)$ for $i \in \{0,1\}$ respectively.
	The group  $\breve{C}^0(\U,\shH^1_{\ZZ_2})$ is the sum of the cohomology groups $H^1(\U_{\sigma},\ZZ_2)$
	for $\sigma$ one of the $10$ $1$-dimensional faces $\sigma^1$ or one of the $10$ $2$-dimensional faces
	$\sigma^2$, and so by
	\eqref{eq:cohomology_edges_2}-\eqref{eq:cohomology_faces_2}, we have 
	\[ \breve{C}^0(\U,\shH^1_{\ZZ_2}) \cong \ZZ_2^{10 \times 12+10 \times 18} = \ZZ_2^{300} \,. \]
    The group $\breve{C}^1(\U,\shH^1_{\ZZ_2})$ is the sum of the cohomology groups
    \[
    H^1(\breve{\pi}^{-1}(\U_{\sigma^1_j} \cap \U_{\sigma^2_k}),\ZZ_2)
    \]
    where $\sigma^1_j$ is an edge of $\sigma^2_k$. By Lemma~\ref{lem:topology_for_intersections} these groups are all isomorphic to $\ZZ_2^8$. Since there are $30$ such intersections, we have that
    \[
    \breve{C}^1(\U,\shH^1_{\ZZ_2}) \cong \ZZ_2^{30\times 8} \cong \ZZ_2^{240}\,.
    \]

To obtain the last two rows of the spectral sequence, note that the corresponding \u{C}ech complex is computed using the analogous computation over the integers, as described in the proof of Theorem~\ref{thm:cech_to_derived}. Since $\breve{H}^3(\U,\shH^0_{\ZZ_2}) \cong \ZZ_2$ and $H^3(\breve{\shL}_\RR,\ZZ_2) \cong \ZZ_2$, the differentials
	\begin{align*}
	 \breve{H}^1(\U,\shH^1_{\ZZ_2})  &\to \breve{H}^3(\U,\shH^0_{\ZZ_2}), \textrm{ and} \\
	 \breve{H}^0(\U,\shH^2_{\ZZ_2})  &\to \breve{H}^3(\U,\shH^0_{\ZZ_2})
	\end{align*}
    on the $E_2$ and $E_3$ pages respectively vanish. Since, $h^1(\breve{\shL}_\RR,\ZZ_2) = h^2(\breve{\shL}_\RR,\ZZ_2) = 101$ by \cite[Corollary~$1.2$]{AP19}, we have that $\dim \coker(\delta) = 40$ and hence $\dim \ker(\delta) = 100$. We show in Theorem~\ref{Relating to integral points} that bases of these vector spaces can be identified with integral points in faces of $\Delta_{\PP^4}$.
    
	The degeneration of the spectral sequence at the $E_2$ page follows from \cite[Corollary~$1.2$]{AP19}. Indeed, by \cite[Example~$4.10$]{AP19} it follows that $h^1(\breve{\shL}_\RR,\ZZ_2) = 101$; this ensures that all morphisms on the $E_2$ and $E_3$ pages vanish.
\end{proof}

We now consider the relationship with the mod $2$ cohomology groups of real Lagrangians and Hodge numbers in more detail.

\begin{theorem}
	\label{Relating to integral points}
	Let $\breve{\pi} \colon \breve{\mathcal{L}}_\RR \to B$ be the fibration on the connected component of the real Lagrangian $\breve{L}_{\RR}$ in the mirror to the quintic as in Lemma~\ref{lem:connected_cmpts}. The terms of the $E_2$ page of the \u{C}ech-to-derived spectral sequence
	\[
	\breve{H}^i(\mathcal{U}, \shH^j_{\ZZ_2}) \Rightarrow H^{i+j}(B;\breve{\pi}_\star\ZZ_2) \cong H^{i+j}(\breve{\shL}_\RR,\ZZ_2)
	\]
	are $\ZZ_2$ vector spaces with dimensions $\dim E_2^{2-p,p}$, equal the total number of integral points in the relative interiors of $p$-dimensional faces of $B\cong\partial \Delta_\PP^4$ for each $p \in \{1,2\}$. Moreover, there is a canonical generating set of $E_2^{2-p,p}$ indexed by these integral points of $\partial \Delta_{\PP^4}$.
\end{theorem}
\begin{proof}
	It is easy to check that there are $6$ points in the relative interior of each of the $10$ two-dimensional faces of $\Delta_{\PP^4}$, and $4$ points in the relative interior of each of the $10$ edges of $\Delta_{\PP^4}$. We then verify that $E^{0,2}_2  \cong \ZZ_2^{60}$, and $E^{1,1}_2  \cong \ZZ_2^{40}$.
	
	Indeed, in the proof of Lemma~\ref{lem:preimage_face} we observed that, restricting to a two-dimensional face $\sigma^2$, the torsion group of $H_1(\breve{\pi}^{-1}(\U_{\sigma^2}),\ZZ)$ is isomorphic to $\ZZ^6_2$. By the universal coefficient theorem, this is isomorphic to the summand $H^2(\breve{\pi}^{-1}(\U_{\sigma^2}),\ZZ_2)$ of 
	\[
    \breve{C}^0(\U,\shH^2_{\ZZ_2}) \cong \breve{H}^0(\U,\shH^2_{\ZZ_2}).
	\]
	This torsion group is generated by elements associated with the hexagonal components $H_i$ for $i \in \{1,\ldots, 6\}$, see Figure~\ref{fig:generating_set}. Note that these hexagonal regions are in canonical bijection with the set of integral points which lie in the relative interior of $\sigma^2$. Moreover, by the universal coefficient theorem, this $\ZZ_2^6$ torsion group is isomorphic to $H^2(\breve{\pi}^{-1}(\U_{\sigma^2}),\ZZ_2)$.
	
	The vector space $E_2^{1,1}$ is the $40$ dimensional cokernel of the map
	\[
	\delta \colon \breve{C}^0(\U,\shH^1_{\ZZ_2}) \to \breve{C}^1(\U,\shH^1_{\ZZ_2}).
	\]
	We recall from the proof of Theorem~\ref{thm:cech_to_derived_Z2} that the vector spaces $\breve{C}^i(\U,\shH^1_{\ZZ_2})$ have dimensions $300$ and $240$ for $i=0$ and $i=1$ respectively. The dual space to $\breve{C}^1(\U,\shH^1_{\ZZ_2})$ is generated by the first homology groups of $\breve{\pi}^{-1}(\U_{\sigma^1} \cap \U_{\sigma^2})$, where $\sigma^i$ is a $i$-dimensional face of $\Delta_{\PP^4}$ for each $i \in \{1,2\}$. We recall that each space $\breve{\pi}^{-1}(\U_{\sigma^1} \cap \U_{\sigma^2})$ retracts onto the disjoint union of $3$ points and $2$ copies of the wedge union of $4$ circles.
	
	Each integral point $p$ in the relative interior of $\sigma^1$ corresponds to a segment between a pair of adjacent (positive) vertices of $\Delta$. In particular, each segment between adjacent (positive) vertices in $\sigma^1$ determines a pair of homology classes in $H_1(\breve{\pi}^{-1}(\U_{\sigma^1} \cap \U_{\sigma^2}),\ZZ_2)$. We let $\chi^1_{p,\sigma^2}$ and $\chi^2_{p,\sigma^2}$ denote this pair of homology classes.
	
	We claim that a basis of the kernel of $\delta^\star$, the linear dual of $\delta$, is given by
	\[
	\chi_p := \sum_{\{\sigma^2 : ~ p \in \sigma^2\}}(\chi^1_{p,\sigma^2} + \chi^2_{p,\sigma^2}),
	\]
    as $p$ varies over the $40$ integral points in the relative interior of the edges of $\partial \Delta_{\PP^4}$. Note that the sum defining $\chi_p$ contains exactly three terms for all $p$.
    
    We first note that, since the classes $\chi^1_{p,\sigma^2}$ and $\chi^2_{p,\sigma^2}$ form a basis of $\breve{C}^1(\U,\shH^1_{\ZZ_2})^\star$ the set of classes $\chi_p$ is linearly independent. To verify that $\chi_p \in \ker \delta^\star$ we first check that $\chi_p$ is in the kernel of the natural map
    \[
    \iota_1 \colon \bigoplus_{\{\sigma^2 : ~ \sigma^1 \subset \sigma^2\}}H_1(\breve{\pi}^{-1}(\U_{\sigma^1} \cap \U_{\sigma^2}),\ZZ_2) \to H_1(\breve{\pi}^{-1}(\U_{\sigma^1}),\ZZ_2)
    \]
    We recall from Lemma~\ref{lem:topology_for_intersections} that $H_1(\breve{\pi}^{-1}(\U_{\sigma^1}),\ZZ_2)$ is generated by the classes of $12$ circles, three of which lie over the segment in $\sigma^1$ containing $p$. Let $\chi^1_p$, $\chi^2_p$, and $\chi^3_p$ denote these homology classes. We note that, for any $i \in \{1,2\}$, $\iota_1(\chi^i_{p,\sigma^2})$ is equal to a class $\chi^j_p$ for some $j \in \{1,2,3\}$.
    Moreover, we have that
    \[
    \iota_1(\chi_p) = \sum_{\{\sigma^2: ~ p \in \sigma^2\}}(\iota_1(\chi^1_{p,\sigma^2}) + \iota_1(\chi^2_{p,\sigma^2})).
    \]
    This is a sum of six cycles in which each of the three classes in $\{\chi^i_p : i \in \{1,2,3\}\}$ appears twice. Hence this sum vanishes modulo $2$.
    
    Finally, we verify that $(\chi^1_{p,\sigma^2} + \chi^2_{p,\sigma^2})$ is in the kernel of the natural map
    \[
    \iota_2 \colon H_1(\breve{\pi}^{-1}(\U_{\sigma^1} \cap \U_{\sigma^2}),\ZZ_2) \to H_1(\breve{\pi}^{-1}(\U_{\sigma^2}),\ZZ_2)
    \]
    for any $p$ and $\sigma^2$ such that $p \in \sigma^2$. However this follows immediately from the fact that the images of the classes of $\chi^1_{p,\sigma^2}$ and $\chi^2_{p,\sigma^2}$ in $H_1(\breve{\pi}^{-1}(\U_{\sigma^2}),\ZZ_2)$ are equal and hence their sum vanishes modulo $2$. These computations verify that $\chi_p$ is in the kernel of $\delta^\star$, and hence the set of points in the relative interior of edges in $\Delta_{\PP^4}$ can be canonically identified with a basis of. 
    \[
    \breve{H}^1(\U,\shH^1_{\ZZ_2}) \cong \ZZ_2^{40}
    \]
\end{proof}

We expect this relationship between mod $2$ cohomology groups of real Lagrangians and Hodge numbers to hold in greater generality, as stated in the following conjecture.

\begin{conjecture}
\label{conj:generalised_example}
Let $X$ be a Calabi--Yau hypersurface in a smooth toric Fano fourfold, and let $\breve{f}\colon \breve{X} \to B$ be a topological torus fibration on the mirror $\breve{X}$ to $X$, as discussed in \S\ref{sec:SYZ Mirrors}. Then, there is a fiber preserving anti-symplectic involution on $\breve{X}$, whose fixed point locus has two connected components, one homeomorphic to $B$ and the other arising as a multi-section $\breve{\pi}:\mathcal{L}_{\RR} \to B$, such that 
there is a \u{C}ech-to-derived spectral sequence for $\breve{\pi}_*\ZZ$, analogous to that appearing in Theorem~\ref{thm:cech_to_derived_Z2}, whose $E_2$ page has the following form
\[
	\xymatrix@R-2pc{
		\ZZ_2^{l_f} & & & \\
		\ZZ_2^{l_e+l_f} & \ZZ_2^{l_e} & & \\
		\ZZ_2 & \ZZ^{a}_2 & \ZZ^{a}_2 & \ZZ_2
	}
\]
where $l_f$ and $l_e$ denote the number of integral points in the relative interiors of two-dimensional faces and edges respectively. Analogously with Theorem~\ref{thm:cech_to_derived_Z2}, there are canonical bijections between integral points and generating sets of $E^{2-p,p}_2$ for $p \in \{1,2\}$. Moreover, this spectral sequence degenerates at the $E_2$ page.
\end{conjecture}

If Conjecture~\ref{conj:generalised_example} holds, the identification of integral points of faces of $B$ with generating sets of entries in the $E_2$ page of the \u{C}ech-to-derived spectral sequence for $\breve{\pi}_\star\ZZ_2$ guarantees, by results of Batyrev--Borisov~\cite{BB,BBci}, a relationship between the dimension of $a := \dim E^{2,0}_2$ and the rank of the connecting homomorphism
\begin{equation}
\label{eq:beta}
\beta \colon H^1(B,R^2 f_\star\ZZ_2) \to H^2(B,R^1 f_\star\ZZ_2)
\end{equation}
defined in \cite{CBM} and described in considerable detail in \cite{AP19}. Applying \cite[Theorem~$1.1$]{AP19}, if the corresponding polytope has $l_v$ vertices, we have that
\begin{equation}
    \label{Eq: beta and points}
a = \dim E_2^{2,0} = \dim \ker (\beta) + l_v - 4.
\end{equation}

We note that the rank of $\beta$, and hence $h^1(\breve{\shL}_\RR,\ZZ_2)$, can be computed easily for hypersurfaces in smooth toric Fano fourfolds using \cite[Theorem~$1.1$]{AP19} and toric geometry. In particular \cite[Theorem~$1.1$]{AP19} allows us to compute the value of $h^1(\breve{\shL}_\RR,\ZZ_2)$ from the square map 
\begin{eqnarray}
\nonumber
    \Square :  H^1(B,R^1 \breve{f}_\star\ZZ_2) & \longrightarrow & H^2(B,R^2 \breve{f}_\star\ZZ_2) \\
\nonumber
    D & \longmapsto & D^2
\end{eqnarray}
in the cohomology ring of $X$. We collect the computation of $h^1(\breve{\shL}_\RR,\ZZ_2)$ for anti-canonical hypersurfaces in each of the $124$ smooth toric Fano fourfolds, as classified in \cite{B99}, in the following result.

\begin{proposition}
    \label{pro:fano_fourfolds}
Let $X$ be an anti-canonical hypersurface in a smooth toric Fano fourfold. 
Then, 
\begin{equation}
     \label{Eq: anticanonical hypersurfaces}
    h^1(\breve{\shL}_\RR,\ZZ_2) - h^{1,1}(\breve{X}) = \dim \ker \Square,
\end{equation}
where $\Square$ is the map $H^2(X,\ZZ_2) \to H^4(X,\ZZ_2)$ given by $\Square \colon D \mapsto D^2$. The values of $h^1(\breve{\shL}_\RR,\ZZ_2)$ for each of the $124$ anti-canonical Calabi--Yau hypersurfaces are as in 
in Table~\ref{tab:fano_fourfolds}.
\end{proposition}
\begin{proof}
    Fix a smooth Fano fourfold $Y$ and an anti-canonical hypersurface $X$ in $Y$. 
    According to \cite{BK} (see Corollary 1.9 and 3.9), the cohomology of $\breve{X}$ with $\ZZ$-coefficients is torsion-free, and so
     it follows from the analysis in \cite[pg 245]{CBM} that $h^1(B, R^1 f_{\star}\ZZ_2)=h^{1,1}(\breve{X})$. The Equation \ref{Eq: anticanonical hypersurfaces} is an immediate corollary of \cite[Theorem~$1.1$]{AP19}. Fix a basis $e_1,\ldots,e_k$ of $H^2(Y,\ZZ_2)$ and note that $H^2(Y,\ZZ_2) \to H^2(X,\ZZ_2)$ is an isomorphism by the Lefschetz hyperplane theorem. Hence the rank $\Square$ is equal to the rank of the $k \times k$ matrix with entries $S_{i,j}$ where
    \[
    S_{i,j} = e_i^2\smile e_j \smile K_Y,
    \]
    where $K_Y$ denotes the canonical class of $Y$ in $H^2(Y,\ZZ_2)$. This matrix can be easily computed using SageMath~\cite{SageMath}. Source code for this computation is included in the supplementary material~\cite{Supplementary}. The column $id$ records the index of each smooth four-dimensional toric Fano variety in the graded rings database \cite{grdb}.

    \begin{table}
        \centering
    \[{\scriptsize
        \arraycolsep=1pt
        \begin{array}{ccccc || ccccc || ccccc}
     id & h^{1,1}(X) & h^0(-K_Y) & \rank \beta & ~d~ & ~id~ & h^{1,1}(X) & h^0(-K_Y) & \rank \beta & ~d~ & ~id~ &  h^{1,1}(X) & h^0(-K_Y) & \rank \beta & d~ \\  \hline
     1 &3& 123& 2& 1& 42 & 4& 70& 4& 0& 83 & 4& 81& 2& 2 \\
     2 &2& 159& 2& 0& 43 & 3& 96& 2& 1& 84 & 5& 65& 3& 2 \\
     3 &4& 114& 2& 2& 44 & 4& 75& 4& 0& 85 & 4& 75& 2& 2 \\
     4 &5& 78& 2& 3& 45 & 3& 90& 2& 1& 86 & 3& 85& 1& 2 \\
     5 &4& 99& 2& 2& 46 & 3& 85& 3& 0& 87 & 4& 77& 3& 1 \\
     6 &4& 96& 3& 1& 47 & 2& 105& 2& 0& 88 & 4& 82& 3& 1 \\
     7 &3& 120& 2& 1& 48 & 5& 92& 0& 5& 89 & 4& 84& 0& 4 \\
     8 &3& 117& 2& 1& 49 & 6& 67& 0& 6& 90 & 5& 69& 0& 5 \\
     9 &4& 81& 2& 2& 50 & 5& 83& 0& 5& 91 & 4& 81& 2& 2 \\
     10 &4& 102& 0& 4& 51 & 4& 99& 0& 4& 92 & 4& 81& 4& 0 \\
     11 &4& 87& 2& 2& 52 & 4& 96& 0& 4& 93 & 4& 74& 3& 1 \\
     12 &3& 114& 0& 3& 53 & 5& 71& 0& 5& 94 & 3& 95& 2& 1 \\
     13 &4& 78& 4& 0& 54 & 5& 85& 3& 2& 95 & 4& 93& 4& 0 \\
     14 &4& 84& 0& 4& 55 & 6& 65& 4& 2& 96 & 5& 72& 4& 1 \\
     15 &4& 90& 3& 1& 56 & 5& 85& 4& 1& 97 & 5& 72& 0& 5 \\
     16 &4& 72& 3& 1& 57 & 6& 65& 4& 2& 98 & 6& 63& 0& 6 \\
     17 &3& 93& 3& 0& 58 & 5& 79& 2& 3& 99 & 5& 70& 2& 3 \\
     18 &3& 108& 3& 0& 59 & 5& 73& 3& 2& 100 & 4& 87& 2& 2 \\
     19 &3& 102& 0& 3& 60 & 4& 93& 2& 2& 101 & 4& 80& 2& 2 \\
     20 &3& 84& 2& 1& 61 & 5& 78& 0& 5& 102 & 4& 78& 2& 2 \\
     21 &2& 120& 2& 0& 62 & 6& 63& 0& 6& 103 & 4& 78& 0& 4 \\
     22 &4& 104& 2& 2& 63 & 6& 61& 4& 2& 104 & 3& 90& 2& 1 \\
     23 &5& 76& 4& 1& 64 & 6& 59& 0& 6& 105 & 3& 101& 2& 1 \\
     24 &4& 92& 4& 0& 65 & 5& 71& 3& 2& 106 & 3& 102& 2& 1 \\
     25 &4& 86& 4& 0& 66 & 5& 71& 5& 0& 107 & 4& 81& 0& 4 \\
     26 &3& 114& 2& 1& 67 & 5& 75& 0& 5& 108 & 4& 75& 0& 4 \\
     27 &4& 87& 3& 1& 68 & 5& 77& 3& 2& 109 & 3& 96& 2& 1 \\
     28 &5& 69& 5& 0& 69 & 5& 67& 0& 5& 110 & 3& 90& 0& 3 \\
     29 &5& 67& 2& 3& 70 & 4& 86& 3& 1& 111 & 3& 99& 2& 1 \\
     30 &4& 70& 4& 0& 71 & 4& 82& 0& 4& 112 &  3& 90& 2& 1 \\
     31 &5& 55& 3& 2& 72 & 4& 87& 0& 4& 113 & 3& 84& 2& 1 \\
     32 &4& 69& 3& 1& 73 & 4& 75& 0& 4& 114 & 3& 84& 3& 0 \\
     33 &5& 66& 3& 2& 74 & 3& 100& 0& 3& 115 & 2& 105& 1& 1 \\
     34 &4& 78& 3& 1& 75 & 6& 64& 2& 4& 116 & 2& 129& 0& 2 \\
     35 &4& 78& 3& 1& 76 & 7& 56& 2& 5& 117 & 3& 93& 0& 3 \\
     36 &4& 81& 2& 2& 77 & 8& 49& 2& 6& 118 & 2& 105& 2& 0 \\
     37 &3& 87& 2& 1& 78 & 6& 63& 2& 4& 119 &4& 81& 0& 4 \\
     38 &4& 81& 1& 3& 79 & 5& 72& 2& 3& 120 &3& 90& 0& 3 \\
     39 &5& 66& 5& 0& 80 & 5& 70& 2& 3& 121 &2& 111& 2& 0 \\
     40 &6& 51& 2& 4& 81 & 4& 76& 2& 2& 122 &2& 105& 0& 2 \\
     41 &3& 90& 3& 0& 82 & 4& 90& 2& 2& 123 & 2& 100& 2& 0 \\
      & & & & & & & & & & 124 & 1 & 126 & 1 & 0
    \end{array}
      }  \]
        \caption{Rank of the map \eqref{eq:beta} for
         anti-canonical hypersurfaces $X\subset Y$ in smooth toric Fano fourfolds $Y$, and $d:= h^1(\breve{L}_\RR,\ZZ_2) - h^{1,1}(X)$.}
        \label{tab:fano_fourfolds}
    \end{table}
\end{proof}

Note that for all real Lagrangians $\breve{L}_\RR \subset \breve{X}$ which appear in Table~\ref{tab:fano_fourfolds}, we have that 
    \[
    h^{1,1}(\breve{X}) \leq h^1(\breve{L}_\RR,\ZZ_2) \leq h^{1,1}(\breve{X}) + h^{1,2}(\breve{X}).
    \]
    These inequalities also follow from \cite[Thm 1]{CBM} if there is no $2$-torsion in the homology of $X$ with $\ZZ$ coefficients.
    Moreover, if $\rank \beta = 0$, then $h^1(\breve{L}_\RR,\ZZ_2) = h^{1,1}(\breve{X}) + h^{1,2}(\breve{X})$, while if $d = 0$ we have that $h^1(\breve{L}_\RR,\ZZ_2) = h^{1,1}(\breve{X})$.



\section{Heegaard splittings}
\label{sec:heegaard_splitting}

In the proof of Proposition~\ref{pro:orientable} we introduced a Heegaard splitting to show that $\breve{\shL}_\RR$ is an orientable $3$-manifold. We describe an algorithm to compute $\pi_1(\breve{\shL}_\RR)$ using this Heegaard splitting and adapt this algorithm in \S\ref{sec:application} to verify calculations made in the proof of Theorem~\ref{thm:cech_to_derived}.

Fix an integral affine manifold $B$, of real dimension $3$, with simple singularities and denote by $\breve{\pi} \colon \breve{\shL}_\RR \to B$ denote the $7$-to-$1$ cover, constructed
analogously as in \S\ref{sec:real_locus}. We let $W_1 \subset B$ be the thickening of the discriminant locus $\Delta$ of $B$ described in the proof of Proposition~\ref{pro:orientable}. In particular, we form a cover of $\Delta \subset B$ by fixing closed subsets $W_v$ and $W_e$ of $B$ for each vertex $v$ and edge $e$ of $\Delta$ which satisfy the conditions given in the proof of Proposition~\ref{pro:orientable} and let
\[
W_1 = \bigcup_v W_v \cup \bigcup_e W_e.
\]


Following a similar topological analysis to that made in \cite[Appendix~A]{AP19}, the preimages of $\breve{\pi}^{-1}(W_v)$ and $\breve{\pi}^{-1}(W_e)$ are disjoint unions of $3$-balls for all vertices $v$ and edges $e$ of $\Delta$. We define $W_2$ to be the closure of the complement of $W_1$ in $B$ and assume throughout this section that $W_2$, and hence $\breve{\pi}^{-1}(W_2)$, is a handlebody, see Remark~\ref{rem:assumption}.


To describe the fundamental group (or first homology group) of $\breve{\shL}_\RR$ we make use of the following elementary observation on the fundamental groups of $3$-manifolds with a Heegaard splitting.

\begin{lemma}
	\label{lem:attaching_discs}
	Given a $3$-manifold $X$ with a genus $g$ Heegaard splitting into (compact) handlebodies $H_1$ and $H_2$, $\pi_1(X)$ is generated by the free group $\pi_1(H_1)$, with relations determined by a collection of meridian discs $\{D_1,\ldots,D_k\}$ of $H_2$ such that the complement of $\bigcup^k_{i=1}{D_i}$ in $H_2$ is a disjoint union of $3$-balls.
\end{lemma}
\begin{proof}
	Writing the interior of $H_2$ as the union of $k$ disks $D_1,\ldots D_k$ and $(k-g+1)$ three-dimensional balls $A_1, \ldots, A_{k-g+1}$, we have that 
	\[
	X = H_1 \cup \bigcup^{k-g+1}_{i=1}A_i \cup \bigcup^k_{i=1}D_i.
	\]
	Recalling that the fundamental group of a CW complex is determined by its $2$-skeleton, removing the $3$-balls $A_i$ does not affect the fundamental group, and $\pi_1(X)$ is isomorphic to
	\[
	\pi_1\big( H_1 \cup \bigcup^k_{i=1}D_i\big).
	\]
	Applying the Seifert--van Kampen theorem $k$ times, $\pi_1(X)$ is isomorphic to the quotient $\pi_1(H_1)$ by the classes of the loops $\partial D_i \subset H_1$ for $i \in \{1,\ldots,k\}$.
\end{proof}

Applying Lemma~\ref{lem:attaching_discs}, we describe the fundamental group of $\breve{\shL}_\RR$ by constructing a generating set of $\pi_1(\breve{\pi}^{-1}(W_2))$ and a collection of meridian discs in $\breve{\pi}^{-1}(W_1)$ whose complement is a disjoint union of $3$-balls. Recalling that we have assumed the space $W_2$ is a handlebody, we fix a cover $\shV$ of $W_2$ by $3$-balls which meet along a pairwise disjoint collection of discs. 

\begin{algorithm}
    \label{alg:fundamental_group}
	We first construct a generating set of $\pi_1(\breve{\pi}^{-1}(W_2))$. Observe that the space $W_2$ retracts onto an embedded graph $\Gamma$ in $B \setminus \Delta$ such that:
	\begin{enumerate}
		\item $\Gamma$ contains a vertex $v$ for each $3$-ball in $\shV$, contained in its interior. 
		\item $\Gamma$ contains an edge $e \subset V_1 \cup V_2$ for every pair of elements $V_1$ and $V_2$ of $\shV$ which intersect in a disc. This edge connects the vertices $v_1$ and $v_2$ corresponding to $V_1$ and $V_2$ respectively and passes through a single point in $V_1 \cap V_2$.
	\end{enumerate}
	
	Note that, as $\breve{\pi}$ is unbranched over $W_2$, $\breve{\pi}^{-1}(\Gamma)$ is a $7$-to-$1$ covering of $\Gamma$. We recall that any connected graph is homotopy equivalent to the wedge union of circles, obtained by contracting a spanning tree to a point. Hence, fixing a spanning tree $T$ in $\breve{\pi}^{-1}(\Gamma)$, a generating set of $\pi_1(\breve{\pi}^{-1}(\Gamma))$ is determined by choosing orientations of the edges of $\breve{\pi}^{-1}(\Gamma) \setminus T$, regarded as loops in the quotient space $\breve{\pi}^{-1}(\Gamma)/T$.
	
	Next we construct a disjoint collection of meridian disks in $\breve{\pi}^{-1}(W_1)$. Recall that, regarding $W_e$ as a solid cylinder thickening $e$, $\breve{\pi}^{-1}(W_e)$ is the disjoint union of five disjoint solid cylinders for any edge $e$ of $\Delta$. The preimage of a disk $D_e$ in $W_e$ which separates $W_e$ into two components is a set of five disks
	\[
	\breve{\pi}^{-1}(D_e) = \{D^i_e : i \in \{1,\ldots,5\}\}
	\]
	which separate each of the five cylinders $\breve{\pi}^{-1}(W_e)$ into two components. The complement of the union of all disks $D^i_e$ retracts onto $\bigcup_v{\breve{\pi}^{-1}(W_v)}$, which is a disjoint collection of $3$-balls. Hence the collection of all disks $D^i_e$ satisfies the conditions of Lemma~\ref{lem:attaching_discs}, and thus determines a set of relations for a presentation of $\pi_1(\breve{\shL}_\RR)$. These relations are obtained explicitly by constructing homotopies in $B \setminus \Delta$ from the circles $\partial D_e$ to cycles in $\Gamma$.
\end{algorithm}

\subsection{Application to the mirror quintic}
\label{sec:application}

We use the Heegaard splitting described above to verify part of the calculations made in the proofs of Theorem~\ref{thm:cech_to_derived} and Theorem~\ref{thm:cech_to_derived_Z2}. In particular, we compute the cohomology groups of the complex $C^\bullet$
\begin{equation}
	\label{eq:cech_complex}
	0 \lra \ZZ^{115} \stackrel{\delta_0}\lra \ZZ^{710} \stackrel{\delta_1}\lra \ZZ^{1190} \stackrel{\delta_2}\lra \ZZ^{595} \lra 0
\end{equation}
which appears in the proof of Theorem~\ref{thm:cech_to_derived}.

\begin{proposition}
	\label{pro:verify_cech_complex}
	The \u{C}ech complex $C^\bullet$ has cohomology groups
	\[
	H^i(C^\bullet) \cong 
	\begin{cases}
	\ZZ & i \in \{0,3\},\\
	0 & i = 1,\\
	\ZZ_2 & i=2.
	\end{cases}
	\]
\end{proposition}

Before proving Proposition~\ref{pro:verify_cech_complex}, we show that $C^\bullet$ computes the cohomology groups of a topological space $\tilde{\shL}_\RR$. To define $\tilde{\shL}_\RR$ we first fix a trivalent graph $\tilde{\Delta}$ in $\partial\Delta_{\PP^4}$, by shrinking the discriminant locus as illustrated in Figure \ref{fig:shrinking delta}.
	\begin{figure}
	    \centering
	    \includegraphics[scale=0.55]{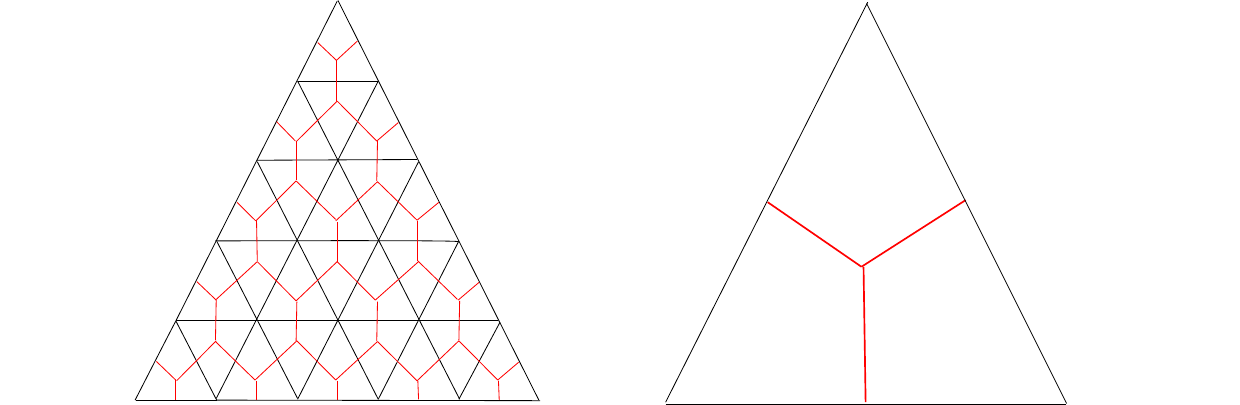}
	    \caption{The discriminant locus $\Delta$ (on the left), and its shrinking to $\tilde{\Delta}$ (on the right) on a triangular face.}
	    \label{fig:shrinking delta}
	\end{figure}
More concretely, $\tilde{\Delta}$ is the union of the cells of the first barycentric subdivison of $\partial\Delta_{\PP^4}$ which are contained in a two-dimensional face of $\Delta_{\PP^4}$ and which do not contain a vertex of $\Delta_{\PP^4}$. Part of $\tilde{\Delta}$ is illustrated in Figure~\ref{fig:heegaard_application}, in which one vertex of $ \Delta_{\PP^4}$ is taken to be at infinity. We define $\tilde{\shL}_\RR$ by fixing a $7$-to-$1$ cover 
\[
\tilde{\pi} \colon \tilde{\shL}_\RR \to \partial \Delta_{\PP^4}
\]
branched over $\tilde{\Delta}$. In particular, we recall that $\partial \Delta_{\PP^4}$ admits an integral affine structure with (non-simple) singularities along $\tilde{\Delta}$, see for example \cite[Proposition~$2.12$]{GrossBB}. Hence we can define a $7$-to-$1$ cover $\tilde{\pi}_0$ over the complement of $\tilde{\Delta}$, following the construction given in \S\ref{sec:real_locus}. While this integral affine structure is not simple, the monodromy of $\tilde{\pi}_0$ around segments of $\tilde{\Delta}$ is identical to that described in \S\ref{sec:real_locus}. That is, we can extend the covering $\tilde{\pi}_0$ to a branched covering $\tilde{\pi}$ over $\partial\Delta_{\PP^4}$ using the local models around positive and negative vertices described in \S\ref{sec:real_locus}.



\begin{lemma}
    The cohomology groups $H^i(C^\bullet)$ are isomorphic to the cohomology groups $H^i(\tilde{\shL}_\RR,\ZZ)$ for all $i \in \{0,1,2,3\}$.
\end{lemma}
\begin{proof}
    Consider the open cover $\U$ used to define the \u{C}ech-to-derived spectral sequence used in Theorem~\ref{thm:cech_to_derived}. By construction, and comparing preimages of elements of the open cover $\U$ (and their intersections) under $\breve{\pi}$ and $\tilde{\pi}$, we observe that
    \begin{equation}
    \label{Eq cech equivalence}
    \breve{C}^i(\U,\tilde{\pi}_\star\ZZ) \cong \breve{C}^i(\U,\breve{\pi}_\star\ZZ) = C^i
    \end{equation}
    for all $i \in \{0,1,2,3\}$, and so the cohomology groups 
    $H^i(C^\bullet)$ are isomorphic to the 
    \u{C}ech cohomology groups $\breve{H}^i(\U,\tilde{\pi}_\star\ZZ)$. Moreover, the open cover $\U$ is acylic for the sheaf $\tilde{\pi}_\star\ZZ$. Hence the \u{C}ech cohomology groups $\breve{H}^i(\U,\tilde{\pi}_\star\ZZ)$ are isomorphic to the sheaf cohomology groups $H^i(B,\tilde{\pi}_\star\ZZ)$. Finally, the sheaf cohomology groups 
    $H^i(B,\tilde{\pi}_\star\ZZ)$ are isomorphic to the cohomology groups $H^i(\tilde{\shL}_\RR,\ZZ)$ by the Leray spectral sequence
    for $\tilde{\pi}$ and using that the fibres of $\tilde{\pi}$ are discrete.

    
\end{proof}

\begin{proof}[Proof of Proposition~\ref{pro:verify_cech_complex}]

	We first fix handlebodies $W_1$ and $W_2$ as in Algorithm~\ref{alg:fundamental_group}. We fix disks $D_i$, for $i \in \{1,\ldots, 12\}$ in $\partial \Delta_{\PP^4}$, three of which are illustrated in Figure~\ref{fig:heegaard_application}. Recall that  $\partial \Delta_{\PP^4}$ is formed by $5$ tetrahedra, has ten $2$-faces, ten $1$-faces, and five vertices, among which one is at infinity, $\infty$. In particular, to a $2$-face spanned by $\{ i,j,k\} \in \{1,2,3,4\}$, as in Figure~\ref{fig:heegaard_application}, there are six other adjacent $2$-faces; three of them lie tetrahedra with vertices $\{ 1,2,3,4\}$ the other three on the tetrahedra with vertices $\{ i.j.k,\infty \}$. The disc $D_1$ in Figure~\ref{fig:heegaard_application}, is a disc containing the vertex labelled by $1$, and enclosing the region that is bounded by part of the discriminant locus lying on the $2$-faces spanned by $\{1,2,3\}$, $\{1,2, \infty\}$ and $\{1,3,\infty\}$. We choose the other discs $D_i$ analogously. Note that the complement of $W_1 \cup \bigcup_{i=1}^{12} D_i$ is a pair of $3$-balls $\shV := \{V_1,V_2\}$. Hence, in this case, the graph $\Gamma$ (as defined in Algorithm~\ref{alg:fundamental_group}) has two vertices and twelve edges between these two vertices, while $\tilde{\pi}^{-1}(\Gamma)$ contains $14$ vertices and $84$ edges. We fix a generating set of 
	\[
	H_1(\tilde{\pi}^{-1}(W_2)) \cong H_1(\tilde{\pi}^{-1}(\Gamma))
	\]
	by noting that a spanning tree $T$ in $\tilde{\pi}^{-1}(\Gamma)$ contains $13$ edges and identifying each of the $71$ edges of $\tilde{\pi}^{-1}(\Gamma) \setminus T$  with a generator of $H_1(\tilde{\pi}^{-1}(\Gamma)))$, as described in Algorithm~\ref{alg:fundamental_group}. In particular, we fix an isomorphism $H_1(\tilde{\pi}^{-1}(\Gamma)) \cong \ZZ^{71}$. Applying Lemma~\ref{lem:attaching_discs}, the space $H_1(\tilde{\shL}_\RR)$ is the quotient of $\ZZ^{71}$ by a subgroup determined by a system of meridian disks of $\tilde{\pi}^{-1}(W_1)$.

	\begin{figure}
	    \centering
	    \includegraphics[scale=0.55]{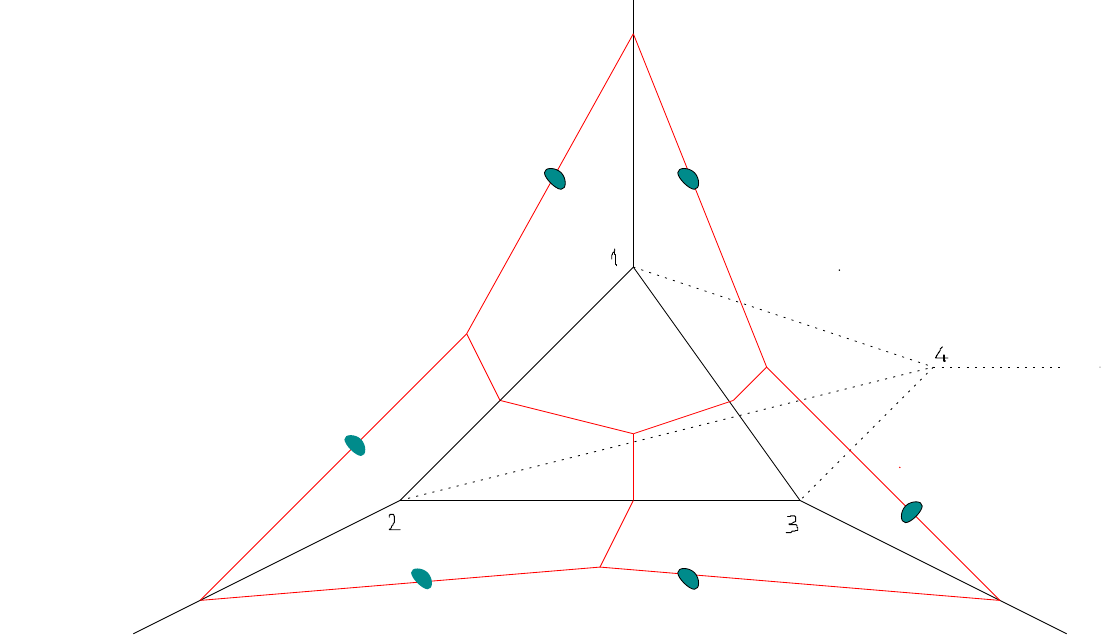}
	    \caption{Three of the discs $D_i$ enclosing regions labelled by $i$ for $i\in \{ 1,2,3 \}$, with boundaries on the discriminant locus, and 6 of the discs $D_e$ depicted in green.}
	    \label{fig:heegaard_application}
	\end{figure}
	
	As described in Algorithm~\ref{alg:fundamental_group}, each edge of $\tilde{\Delta}$ determines a collection of five disjoint meridian disks in $\tilde{\pi}^{-1}(W_1)$. We consider $24$ meridian discs $D_e$, associated to $24$ (of the total $30$) edges $e$ of $\tilde{\Delta}$, six of which are illustrated in Figure~\ref{fig:heegaard_application}. The preimage of the union of these $24$ disks is a disjoint union of $120$ disks $D^i_e$, where $i \in \{1,\ldots, 5\}$. It is straightforward to verify that the complement of these discs in $\tilde{\pi}^{-1}(W_1)$ is a disjoint collection of $3$-balls.
	
	Expressing an orientation of each cycle $\partial D^i_e \subset \tilde{\pi}^{-1}(W_2)$ in terms of the generating set of $H_1(\tilde{\pi}^{-1}(W_2)) \cong \ZZ^{71}$ described above, we can express $H_1(\tilde{\shL}_\RR)$ as the quotient of $\ZZ^{71}$ by a subgroup generated by $120$ elements. Following the computation of this matrix, source code for which is included in the supplementary material, this quotient group is $\ZZ_2$, and hence $H_1(\tilde{\shL}_\RR,\ZZ) \cong \ZZ_2$. Since $\tilde{\shL}_\RR$ is connected and orientable, $H^i(\tilde{\shL}_\RR) \cong \ZZ$ for $i \in \{0,3\}$. Moreover, by the universal coefficient theorem and Poincar\'e duality, we have that $H^1(\tilde{\shL}_\RR) \cong 0$ and $H^2(\tilde{\shL}_\RR) \cong \ZZ_2$.
\end{proof}

\bibliographystyle{plain}
\bibliography{bibliography}

\end{document}